\documentclass[10pt, twoside]{article}
\usepackage{amsfonts}
\usepackage{stmaryrd}
\usepackage{amssymb}
\usepackage{amsthm}
\usepackage{latexsym,amsmath}
\usepackage{graphicx}
\usepackage{dsfont}
\usepackage{enumerate,enumitem}
\usepackage{color}

\usepackage{float}

\usepackage[colorlinks=true,
            linkcolor=blue,
            citecolor=blue]{hyperref}

\usepackage[ruled,vlined]{algorithm2e}

\usepackage{fancyhdr}
\newtheorem{theorem}{Theorem}[section]
\newtheorem{lem}{Lemma}[section]
\newtheorem{pro}{Proposition}[section]
\newtheorem{cor}{Corollary}[section]
\newtheorem{rem}{Remark}[section]
\newtheorem{rems}{Remarks}[section]
\newtheorem{ex}{Example}[section]
\newtheorem{defi}{Definition}[section]
\newtheorem{hyp}{Assumption}[section]

\newcommand{\be}{\begin{equation}}
\newcommand{\ee}{\end{equation}}
\newcommand{\bde}{\begin{displaymath}}
\newcommand{\ede}{\end{displaymath}}
\newcommand{\beq}{\begin{eqnarray*}}
\newcommand{\eeq}{\end{eqnarray*}}
\newcommand{\beqa}{\begin{eqnarray}}
\newcommand{\eeqa}{\end{eqnarray}}
\newcommand{\bel }{\left\{\begin{array}{ll}}
\newcommand{\eel}{\cr \end{array} \right.}

\newcommand{\E}{\mathbb{E}}

\def\wt{\widetilde}

\def\F{{\cal F}}

\def\R{{\mathbb R}}
\def\rr{{\mathbb R}}

\def\P{{\mathbb P}}

\def\I{\mathbf{1}}

\newcommand{\bt}{\begin{theorem}}
\newcommand{\et}{\end{theorem}}
\newcommand{\bl}{\begin{lem}}
\newcommand{\el}{\end{lem}}
\newcommand{\bp}{\begin{pro}}
\newcommand{\ep}{\end{pro}}
\newcommand{\bcor}{\begin{cor}}
\newcommand{\ecor}{\end{cor}}

\newcommand{\bd}{\begin{defi} \rm }
\newcommand{\ed}{\end{defi}}
\newcommand{\brem }{\begin{rem} \rm }
\newcommand{\erem }{\end{rem}}
\newcommand{\brems }{\begin{rems} \rm }
\newcommand{\erems }{\end{rems}}
\newcommand{\bhyp }{\begin{hyp} \rm }
\newcommand{\ehyp }{\end{hyp}}
\newcommand{\bex}{\begin{ex} \rm }
\newcommand{\eex}{\end{ex}}

\title{{A positivity-preserving numerical scheme \\for the alpha-CEV process.}}
\author{Libo Li and Guanting Liu
\\ School of Mathematics and Statistics
\\ University of New South Wales
\\ NSW 2052, Australia
}

\begin{document}

\maketitle
\begin{abstract}
In this article, we present a method to construct a positivity-preserving numerical scheme for a jump-extended CEV (Constant Elasticity of Variance) process, whose jumps are governed by a spectrally positive $\alpha$-stable process with $\alpha \in (1,2)$. The numerical scheme is obtained by making the diffusion coefficient $x^\gamma$, where $\gamma \in (\frac{1}{2},1)$, partially implicit and then finding the appropriate adjustment factor. We show that, for sufficiently small step size, the proposed scheme converges and theoretically achieves a strong convergence rate of at least $\frac{1}{2}\left(\frac{\alpha_-}{2} \wedge \frac{1}{\alpha}\wedge \rho\right)$, where $\rho \in (\frac{1}{2},1)$ is the H\"older exponent of the jump coefficient $x^\rho$ and the constant $\alpha_- < \alpha$ can be chosen arbitrarily close to $\alpha \in (1,2)$.
\end{abstract}

\vskip 5pt \noindent {\small {\bf Key words and phrases.}  Euler-Maruyama scheme, positivity preserving implicit scheme, CBI process, alpha-CIR process and alpha-CEV process.

\vfill\break 

\section*{Introduction}
The Cox-Ingersoll-Ross (CIR) process and the constant elasticity of variance (CEV) process have been used extensively in financial applications such as the modelling of interest rates, default rates, and stochastic volatility, e.g. Duffie et al. \cite{Duffie1, Duffie2} and Alfonsi and Brigo \cite{AB}.
    As a result, the study of strong approximation schemes and positivity-preserving strong approximation schemes for CIR and CEV processes has received a great deal of attention in the literature. We mention here the works of
Alfonsi \cite{Alf1, Alf2},
Berkaoui et al. \cite{BBD},
Alfonsi and Brigo \cite{AB},
Dereich et al. \cite{Dereichetal},
Neuenkirch and Szpruch \cite{NS},
and more recently, Bossy and Olivero \cite{BO},
Chassagneux et al. \cite{CJM},
Hefter and Herzwurm \cite{HH},
Hefter and Jentzen \cite{HJ},
and Cozma and Reisinger \cite{CR}. The literature on this topic is vast and interested readers can refer to the references within the above-mentioned works.

In this work, we study a positivity-preserving numerical scheme for an CEV process with the addition of jumps which are governed by a stable process. More specifically, given initial point $x_0>0$, $\gamma \in (\frac{1}{2},1]$ and $\rho \in ( 1 - \frac{1}{\alpha},1]$, we consider the non-negative solution to the following stochastic differential equation
\begin{align}\label{CIR_0}
dX_t & = \left(a - kX_t\right)dt + \sigma_1 (X_t)^\gamma dW_t + \sigma_2 (X_{t-})^{\rho}dZ_t, \\
X_0  & = x_0\nonumber 
\end{align}
where $a$, $\sigma_1, \sigma_2$ are non-negative and $k \in \rr$. The {\it diffusion coefficient} and the {\it jump coefficient} are given by $x\mapsto \sigma_1 (x^+)^{\gamma}$ and $x\mapsto \sigma_2 (x^+)^\rho$ respectively. The process $W$ is a Brownian motion,
and for $\alpha \in (1,2)$, the process $Z$ is a spectrally positive strictly $\alpha$-stable process independent of $W$.
In other words, the process $Z$ is a L\'evy process with the characteristic triple $(0,\nu,\gamma_0)$, where the L\'evy measure $\nu$ is given $$\nu(dx) = -\frac{1}{\cos(\pi\alpha/2)\Gamma(-\alpha)}x^{-1-\alpha} \I_{(0,\infty)}(x) dx$$ and the drift is given by $\gamma_0 = -\int^\infty_1 x \nu(dx)$.
Also, for the strictly $\alpha$-stable process $Z$,
we have the alternative form given by 
\begin{gather*}
Z_t = \int^t_0\int_0^\infty z \widetilde N(dz,ds),
\end{gather*}
where $\widetilde N$ is the associated compensated Poisson random measure of the L\'evy measure $\nu$. This class of models was first studied in the context of continuous-state branching processes with interaction or/and immigration, see for example, Li and Mytnik \cite{LiMytnik}, Fu and Li \cite{FuLi}. In general, the above SDE exhibits a unique non-negative strong solution for any integrable compensated spectrally one-sided L\'evy process $Z$, see for example Theorem 2.3 in \cite{LiMytnik}.

This work is motivated by the simulation of the so-called {\it alpha-CIR process}. 
More specifically, the alpha-CIR process corresponds to the solution to \eqref{CIR_0} when $\gamma = \frac{1}{2}$ and $\rho = \frac{1}{\alpha}$ and was recently introduced to the mathematical finance literature, in Jiao et al. \cite{JMS, JMSS, JMSZ}, to model sovereign interest rates, power and energy markets and the volatility of financial stocks. The authors in \cite{JMS, JMSS, JMSZ} argue that the model can capture persistent low interest rates, self-exciting behaviours and the large jumps exhibited by sovereign interest rates and power markets and is also suitable for the modelling of stochastic volatility. Given this, we take the liberty to refer to the solution to \eqref{CIR_0} when $\gamma\in (\frac{1}{2},1]$ and $\rho = \frac{1}{\alpha} \in (1-\frac{1}{\alpha},1]$ as the {\it alpha-CEV process} and will examine it as a special case of \eqref{CIR_0}. Although the methodologies represented below can be applied to the case $\gamma = \frac{1}{2}$, however, since the probabilistic nature of the two cases is very different, we will focus only on the case $\gamma > \frac{1}{2}$. The case $\gamma= \frac{1}{2}$ will be studied in more detail in a separate work.

In the current literature, numerical schemes for jump-extended CEV and jump-extended  CIR models have started to receive increasing attention, and we refer to Yang and Wang \cite{YangWang}, Fatemion Aghdas et al. \cite{FHT} and Stamatiou \cite{ST}.
    However, to the best of our knowledge,
for the jump-extended CEV process, the existing results have all focused on the case of Poisson jumps (finite activity jumps), and results on positivity-preserving strong approximation schemes in the case of infinite activity jumps have only appeared in our previous work, Li and Taguchi \cite{LiTaguchi}, in the case of the alpha-CIR process. For weak approximations of SDEs with non-Lipschitz coefficients driven by $\alpha$-stable processes we mention the recent work of Gottwald and Melbourne \cite{GM}.

Numerically speaking, the advantage of our scheme given in \eqref{squareddouble} is similar to that of the CIR process in Alfonsi \cite{Alf1} and the alpha-CIR process in Li and Taguchi \cite{LiTaguchi}, that is, the scheme is obtained by solving at each step a quadratic equation and all quantities involved can be easily simulated without approximation. Therefore, even in the diffusion case ($\sigma_2 = 0$), we do not need to solve for the positive root of a non-linear equation as done in \cite{Alf2, Dereichetal, NS}, where the scheme was obtained through combining the Lamperti transform and the backward Euler scheme. We mention also that the symmetrized scheme developed in  Berkaoui et al \cite{BBD}, Bossy and Olivero \cite{BO}, Bossy and Diop \cite{BossyDiop} and  Bossy et al. \cite{BGT} can potentially be applied here and the symmetrized scheme also has the advantage of not having to solve a non-linear equation at each step. However, in the case of infinite activity jumps, the local time techniques used in the proof of convergence do not appear to translate well into our setting. This is because, in addition to the continuous local time, one needs to compensate the reflected jumps and one quickly faces integrability issues. Of course, this is outside the scope of the current work and will not be studied.

We stress that although the derived scheme in \eqref{squareddouble} might appear similar to the one given in Alfonsi \cite{Alf1} and Li and Taguchi \cite{LiTaguchi}, it was initially unclear how existing techniques in the literature can be adapted to obtain a desirable scheme when $\gamma \neq \frac{1}{2}$. We remark that it was mentioned on page 4 of \cite{Alf1} that, it appears, a quadratic equation can only be obtained if $\gamma \in \{\frac{1}{2},1\}$. Also, due to the presence of infinite activity jumps, jump-adapted schemes devised by combining the Lamperti transform and the backward Euler scheme are not feasible. Hence there was a real need to search for a simulation scheme that does not make use of transformations.
The key idea behind the derivation of our discretization scheme is to make the diffusion coefficient \emph{partially implicit} and to identify an appropriate adjustment factor.
By doing so, we can obtain a positivity-preserving scheme, by solving quadratic equations, for both the alpha-CIR and the alpha-CEV.

The main difference between our scheme in the jump-extended case and the diffusion case is that, with jumps, the discriminant of proposed quadratic equation in \eqref{squareddouble} is not guaranteed to be positive. In view of this, we have to further modify it by taking the positive part of the discriminant to ensure that the scheme is well-defined on the whole time interval. Therefore our main convergence result relies on a key technical lemma, namely Lemma \ref{probability:D:negative}, which states that the probability for the discriminant at each grid point to be negative is exponentially small with respect to the step size.

The article is structured as follows. In Section \ref{Section:CEV:process}, we present the derivation of our scheme and give some auxiliary lemmas and estimates under Assumption \ref{Assumption:1}, Assumption \ref{Assumption:2} and Assumption \ref{Assumption:3} below. The proofs of these auxiliary lemmas and estimates are given in the appendix. Under these assumptions, we show in Theorem \ref{Strong:Convergence}, that the strong rate of convergence in the jump-extended CEV process given in \eqref{CIR_0} is at least $\frac{1}{2}\left(\frac{\alpha_-}{2}\wedge \frac{1}{\alpha}\wedge \rho\right)$ where $\alpha_-$ is a constant smaller than $\alpha$ and can be chosen arbitrarily close to $\alpha$. The derivation of the convergence rate is done through a careful split of the proposed scheme, making use of the martingale representation theorem in the L\'evy filtration, and applying the Yamada-Watanabe approximation technique. It is important to point out that the way in which we split the scheme in \eqref{continuous:extension} is crucial in achieving a rate of convergence which is higher than that of the Euler-Maruyama scheme, see \cite{FL, LiTaguchiEuler}, in the case of the alpha-CEV, that is $\rho = \frac{1}{\alpha}$.  See Remark \ref{Remark:splitting:cont:ext} and Remark \ref{remark1.5} for more details. 

We mention that one can show, under only Assumption \ref{Assumption:1} and Assumption \ref{Assumption:2}, that the proposed numerical scheme convergences. However, the rate of convergence would depend on $\gamma$ and is logarithmic when $\gamma = \frac{1}{2}$. To improve the rate of convergence, by removing its dependence on $\gamma$, we further assume Assumption \ref{Assumption:3} and show in Lemma \ref{CEV:positivity} that the jump-extended CEV process is strictly positive. This allows one to compute, in Lemma \ref{CEV:inverse:moments}, inverse moment estimates for the jump-extended CEV process and to make use of a technique from Berkaoui et al. \cite{BBD} to improve the rate of convergence. Although not discussed in this work, this technique enables us to improve, under the condition $a- \sigma^2/2 > 0$, the convergence rate in the alpha-CIR case from logarithmic to polynomial. Another technical point that we dealt with is in Lemma \ref{scheme:beta:moments}, where we show the existence of the $\beta$-moment for the numerical scheme, for $\beta \in [1,\alpha)$.    This result is fundamental in removing the boundedness assumption on the jump coefficient made in \cite{Ha, HaTsu, LiTaguchi, LiTaguchiEuler, ZH} and thus the removal of the truncation step and the restriction $\alpha > \sqrt{2}$ in \cite{LiTaguchi}.
Even in the case of the Euler-Maruyama scheme, this integrability issue was only recently examined in Frikha and Li \cite{FL}, where a mean-field extension of the equation from Li and Mytnik \cite{LiMytnik} together with the corresponding propagation of chaos property and the Euler-Maruyama scheme were studied.

Finally, we use $C,C',C_1,C_2,C_T,c\dots$etc, to denote positive constants, which may change from line to line. Given a stopping time $\tau$ and a process $X$, the stopped process is denoted by $X^\tau$ and, for $0\leq t_i < t_{i+1}$, we set $\Delta X_{t_i} := X_{t_{i+1}} - X_{t_i}$. For general results on L\'evy processes we refer to Sato \cite{Sato} and Applebaum \cite{Applebaum}.

\section{The positivity-preserving scheme}\label{Section:CEV:process}
In the following, we consider the equal distanced grid $\pi: 0 = t_0<t_1<...<t_n=T$ with step size $\Delta t = T/n$.
The design of this scheme is inspired by Alfonsi's work on the diffusion CIR process in \cite{Alf1}
which was later extended to the alpha-CIR in Li and Taguchi \cite{LiTaguchi}. 
We describe below the main idea behind our scheme. To this end, we start with the Euler-Maruyama scheme
\begin{equation*}
\Delta X_{t_i} = (a-kX_{t_i})\Delta t + \sigma_1 (X_{t_i})^\gamma\Delta W_{t_i} + \sigma_2 (X_{t_i})^{\rho}\Delta Z_{t_i}.
\end{equation*}
To proceed, we make the diffusion coefficient {\it partially implicit}, and consider for $i = 0,1,2,..., n-1$,
\begin{align*}
\Delta X_{t_i}
&= (a-kX_{t_i})\Delta t + \sigma_1 (X_{t_{i+1}})^{\frac{1}{2}}(X_{t_i})^{\gamma-\frac{1}{2}}\Delta W_{t_i} + \sigma_2 (X_{t_i})^\rho\Delta Z_{t_i} - \sigma_1 ((X_{t_{i+1}})^{\frac{1}{2}} - (X_{t_{i}})^{\frac{1}{2}} )(X_{t_i})^{\gamma-\frac{1}{2}}\Delta W_{t_i}.
\end{align*}
Then by summing over the index $i$ and supposing that the scheme is positive, we see that the last term, or the adjustment factor, is given by
\begin{align*}
\sum_{i=0}^n \sigma_1 ((X_{t_{i+1}})^{\frac{1}{2}} - (X_{t_{i}})^{\frac{1}{2}} )(X_{t_i})^{\gamma-\frac{1}{2}}\Delta W_{t_i} \approx \sigma_1\int_0^t X^{\gamma-\frac{1}{2}}_s d\langle X^\frac{1}{2}, W\rangle_s = \frac{\sigma_1^2}{2}\int^t_0 X_s^{2\gamma-1} ds.
\end{align*}
The above computations suggest that we should consider the following implicit scheme:  $X_{t_0}  = x_0$ and
	\begin{align}
	\Delta X_{t_{i}}& = (a-kX_{t_{i+1}})\Delta t + \sigma_1 (X_{t_{i+1}})^{\frac{1}{2}}(X_{t_i})^{\gamma-\frac{1}{2}}\Delta W_{t_i} + \sigma_2 (X_{t_i})^{\rho}\Delta Z_{t_i} - \frac{\sigma_1^2}{2}X_{t_i}^{2\gamma-1}\Delta t. \label{e2}
	\end{align}
For every $i =0,1,..., n-1$, by setting $X_{t_{i+1}}^{1/2} = x$ and rearranging equation \eqref{e2}, one can obtain a quadratic equation in $x$ given by
\begin{align}\label{quadratic:equation}
(1+k\Delta t)x^2 - \sigma_1 {X}_{t_i}^{\gamma-\frac{1}{2}} \Delta W_{t_i} x - ({X}_{t_i} + (a - \frac{\sigma_1^2}{2}X_{t_i}^{2\gamma-1})\Delta t + \sigma_2 {X}_{t_{i}}^\rho \Delta Z_{t_i})=0.
\end{align}
Therefore, if the discriminant process $D = (D_{t_i})_{i=0,\dots, n-1}$ given by
\begin{align}
		D_{t_{i}}
	&=
		\sigma_1^2(X^{n}_{t_i})^{2\gamma-1} (\Delta W_{t_i})^2  +
		4(1+k\Delta t)(		X^{n}_{t_i} 
		+ \big(a - \frac{\sigma_1^2}{2}(X^{n}_{t_i})^{2\gamma-1}\big)\Delta t
		+ \sigma_2  ( X^{n}_{t_i})^{\rho}\Delta Z_{t_i}) \label{D}
\end{align}
is non-negative at each grid point then a positivity-preserving scheme can be obtained by taking the positive solution to the quadratic equation in \eqref{quadratic:equation}.

However, the presence of $\Delta Z_{t_i}$ in the discriminant $D_{t_i}$ implies that $\mathbb{P}(D_{t_i} < 0) > 0$ and the implicit scheme in \eqref{e2} is not well defined on the whole time interval. In view of this, we further modify the implicit scheme given in \eqref{e2} by taking the positive part of the discriminant to guarantee, at each step, the existence of a unique positive root. To this end, we propose the following positivity-preserving numerical scheme $X^n$:
\begin{align}
		X^{n}_{t_{i+1}}
		&=
		\left[\frac{\sigma_1(X^{n}_{t_i})^{\gamma-\frac{1}{2}} \Delta W_{t_i}
		+ \sqrt{D_{t_{i}}^+} 
		}{2(1+k\Delta t)}\right]^2, \quad i = 0,\dots, n-1\nonumber\\
		X^n_{t_0} & = x_0.
		\label{squareddouble}
	\end{align}

\noindent In the diffusion case, $\sigma_2 =0$, one can show that, given the step size $\Delta t$ is sufficiently small (see Lemma \ref{probability:D:negative}), the discriminant process $D$ associated with \eqref{quadratic:equation} is almost surely non-negative and hence a unique positive solution exists without taking the positive part of $D$ as done in \eqref{squareddouble}. However, for $\sigma_2 >0$, due to the presence of $Z$ which is a compensated spectrally positive $\alpha$-stable process, that is a L\'evy process of infinite variation (L\'evy process of Type C), we have $\mathbb{P}(D_{t_{i}} < 0 ) > 0$. In this case it is not possible to select parameters $a$, $k$, $\sigma_1$, $\sigma_2$ and $\Delta t$ in such a way that the discriminant process $D$ is almost surely non-negative as done in the diffusion case. This observation is due to the fact that the support of $Z$ is not bounded below (see Theorem 24.10 (iii) in Sato \cite{Sato}). We expand on this comment in more detail in the remark below.

\brem \label{Rem_0}
In the case where $Z$ has finite activity (Type A) or has infinite activity and is of finite variation (Type B), see Definition 11.9 in \cite{Sato}, it is possible to find a set of conditions on the parameters  $a$, $k$, $\sigma_1$, $\sigma_2$ and $\Delta t$ to ensure that the discriminant process $D$ is non-negative. To see this, suppose that the support of the L\'evy measure $\nu$ contains $0$ (see page 148 of \cite{Sato} for the definition and properties of the support of a measure). From Theorem 24.10 (iii) in Sato \cite{Sato}, we know that the support of $\Delta Z_{t_i}$ is almost surely contained in $[\gamma_0 \Delta t_i,\infty)$, where the drift $\gamma_0$ is given by $\gamma_0 = -\int^\infty_1 x\nu(dx)$.
Therefore, we obtain 
\begin{align*}
\frac{D_{t_i}}{4(1+k\Delta t)} 
& \geq X^{n}_{t_i} + \Big(a - \frac{\sigma_1^2}{2}(X^{n}_{t_i})^{2\gamma-1}\Big)\Delta t_i + \sigma_2 |X^{n}_{t_i}|^\rho \Delta Z_{t_i} 
\geq X^{n}_{t_i} + \Big(a - \frac{\sigma_1^2}{2}(X^{n}_{t_i})^{2\gamma-1}
	+ \sigma_2 |X^{n}_{t_i}|^\rho\gamma_0\Big)\Delta t_i.
\end{align*}
By considering the convex function $z \mapsto |{X}^{n}_{t_i}|^z$ over the domain $[0,1]$ and applying Jensen's inequality, we see that $|{X}^{n}_{t_i}|^{z} \leq \left(1-z\right) + z{X}^{n}_{t_i}$ for all $z$ in the interval $[0,1]$. Since $\rho \in (1-\frac{1}{\alpha},1]$ lies in this interval, this inequality holds for $z=\rho$. Applying this inequality to the above expression gives, 
\begin{align*}
{X}^{n}_{t_i} +  	\Big(a - \frac{\sigma_1^2}{2}(X^n_{t_i})^{2\gamma-1}
	+
\sigma_2 |X^{n}_{t_i}|^\rho\gamma_0 \Big)\Delta t_i 
&\geq  \Big(1 
	+
(\sigma_2 \gamma_0\rho - \frac{\sigma^2}{2})\Delta t_i \Big){X}^{n}_{t_i} + \Big(a - \frac{\sigma_1^2}{2}
	+
\sigma_2\gamma_0\left(1-\rho\right) \Big)\Delta t_i.
\end{align*}
Hence, we arrive at the following set of sufficient conditions on the parameters
\begin{align}
1 + \Big(\sigma_2 \gamma_0\rho - \frac{\sigma^2}{2}\Big)\Delta t_i >0
	\quad \mathrm{and} \quad
a - \frac{\sigma_1^2}{2}
	+
\sigma_2\gamma_0\left(1-\rho\right)>0.\label{poissoncond}
\end{align}
Therefore, if $Z$ is the compensated Poisson process and condition \eqref{poissoncond} holds, then the implicit scheme given in \eqref{e2} gives rise to a positivity-preserving numerical scheme.
\erem

The aim of the rest of this work is to compute the strong rate of convergence for the scheme given in \eqref{squareddouble}. From this point onward we suppose that the following assumptions hold.
\bhyp\label{Assumption:1}
We make the following assumptions on the model parameters,
\begin{enumerate}[nosep]
  \setlength{\itemsep}{1pt}
\item[(i)] $a$, $\sigma_1$ and $\sigma_2$ are non-negative constants, $x_0>0$ and $k\in \mathbb{R}$.
\item[(ii)] $\alpha \in (1,2)$.
\item[(iii)] $\gamma \in (\frac{1}{2}, 1)$ and $2\gamma  < \alpha$.
\item[(iv)] $\rho \in (1-\frac{1}{\alpha}, 1)$.
\end{enumerate}
\ehyp

\bhyp\label{Assumption:2}
The step size $\Delta t = T/n$ is sufficiently small (or $n$ sufficiently large) so there exists $\kappa_0\in(0,1)$ that $\kappa_n:= 1+k\Delta t \geq \kappa_0$. In the case $\sigma_2 >0$ we require also
\begin{align*}
\Delta t^\frac{2\gamma-1}{2-2\gamma}\frac{2}{\sigma_1^2}\frac{2(1-\gamma)}{(2\gamma-1)}\Big[\frac{\sigma_1^2(2\gamma-1)}{2}\Big]^\frac{1}{2-2\gamma} -\frac{2a}{\sigma_1^2} < 0 \quad \mathrm{and}\quad  \Delta t\,\frac{2a}{\sigma_1^2}\,
\frac{(2\gamma-1)}{2(1-\gamma)}
\Big[\frac{\sigma^2_1(2-2\gamma)}{2a}\Big]^{\frac{1}{2\gamma-1}}\, -\frac{2}{\sigma_1^2} < 0.
\end{align*}
While, in the case $\sigma_2 = 0$, we require $\Delta t \leq \sigma_1^{-\frac{2}{2\gamma-1}}\frac{2}{2\gamma-1}\big[\frac{a}{1-\gamma}\big]^\frac{2(1-\gamma)}{2\gamma-1}$.
\ehyp
We point out that Assumption \ref{Assumption:1} and Assumption \ref{Assumption:2} are the minimal assumptions required to show the convergence of the proposed positivity-preserving scheme and are, without further mention, assumed throughout the rest of the article. The condition $2\gamma<\alpha$ appears naturally when estimating the quadratic variation term in the approximation error since an $\alpha$-stable process $Z$ (and hence the process $X$) can only have $\beta$-moments for $\beta<\alpha$. Note that in the diffusion setting, we effectively have $\alpha = 2$ (as everything is square integrable) and the condition $2\gamma < \alpha$ is automatically satisfied. Also, we exclude the case $\gamma = 1$ since it would contradict the assumption $2\gamma < \alpha$. The assumption on the step size $\Delta t$ in Assumption \ref{Assumption:3} is used to control the probability that the discriminant process $D$ is negative in Lemma \ref{probability:D:negative}, which is key to obtaining our convergence results.

Furthermore, in order to improve the rate of convergence, by removing its dependence on $\gamma$, through making use of inverse moment estimates of $X$, we make the following additional assumptions which are only used in the proof of Lemma \ref{CEV:positivity} and Lemma \ref{CEV:inverse:moments}.

\bhyp\label{Assumption:3} 
We suppose $k > 0 $ and $\rho \in (\frac{1}{2}, 1)$.
\ehyp
\noindent From an application point of view, Assumption \ref{Assumption:3} is not very restrictive, since $k > 0$ represents the speed of mean reversion and one appears particularly interested in the case $\rho = \frac{1}{\alpha} \in (\frac{1}{2},1)$ as discussed in \cite{JMS, JMSS, JMSZ}. 


\subsection{Continuous time dynamics}

By expanding the quadratic equation in \eqref{squareddouble} and using the equality $D_{t_{i}}^+ = D_{t_{i}} + D^-_{t_{i}}$ we obtain
\begin{align*}
		X^{n}_{t_{i+1}}
		& = \frac{\sigma_1^2(X^{n}_{t_i})^{2\gamma-1} (\Delta W_{t_i})^2}{4(1+k\Delta t)^2} + \frac{2 \sigma_1(X^{n}_{t_i})^{\gamma-\frac{1}{2}} \Delta W_{t_i} \sqrt{D_{t_{i}}^+}}{4(1+k\Delta t)^2}  + \frac{D_{t_{i}}}{4(1+k\Delta t)^2} + \frac{D_{t_{i}}^{-}}{4(1+k\Delta t)^2}
\end{align*}
where the discriminate process $D$ at time $t_{i}$ is given by
\begin{align*}
		D_{t_{i}}
	&=
		\sigma_1^2(X^{n}_{t_i})^{2\gamma-1} (\Delta W_{t_i})^2  +
		4(1+k\Delta t)(		X^{n}_{t_i} 
		+ \big(a - \frac{\sigma_1^2}{2}(X^{n}_{t_i})^{2\gamma-1}\big)\Delta t
		+ \sigma_2  ( X^{n}_{t_i})^{\rho}\Delta Z_{t_i}).
\end{align*}
After collecting the appropriate terms,
we obtain the following expression for the scheme:
	\begin{equation}\label{Scheme:expanded}
	X_{t_{i+1}}^{n}
	= X_{t_i}^n +
	(a-k_n X_{t_i}^{n})\Delta t + \sigma_1
	(X^{n}_{t_i})^{\gamma}  \Delta W_{t_i} 
	+ \sigma^n_2 (X_{t_i}^{n})^{\rho} \Delta Z_{t_i} + \Delta R^n_{t_i},
	\end{equation}
	where $k_n := k/(1+k\Delta t)$ and $\sigma_2^n := \sigma_2/(1+k\Delta t)$. The remainder term $\Delta R^n_{t_i}$ is given by
\begin{align*}
		\Delta R^n_{t_i}
	& =  \frac{\sigma_1^2 (X_{t_i}^{n})^{2\gamma-1}}{2}\Big[\frac{(\Delta W_{t_i})^2}{(1+k\Delta t)^2} - \frac{\Delta t}{1+k\Delta t} \Big]
		\nonumber \\
	&\quad+ a\Delta t\Big[ \frac{1}{1+k\Delta t} - 1 \Big]
		+\Delta M_{t_i}^n
			-\sigma_1(X^{n}_{t_i})^{\gamma}   \Delta W_{t_i}
		+ \frac{1}{2(1+k\Delta t)^2}D^-_{t_{i}},
\end{align*}
	where the term $\Delta M_{t_i}^n$ is given by
	\begin{equation*}
	\Delta M_{t_i}^n = 
	\frac{\sigma_1 (X_{t_i}^{n})^{\gamma-\frac{1}{2}}\Delta W_{t_i}}{2(1+k\Delta t)^2}
	\sqrt{D_{t_{i}}^+}
	\end{equation*}
and is a martingale increment since the term $D_{t_{i}}^+$ is an even function in $\Delta W_{t_i}$. The semimartingale decomposition of the term $\Delta R_{t_i}^n$ is then given by $\Delta R_{t_i}^n = \Delta \widehat{M}^n_{t_i}+ \Delta \overline{M}^n_{t_i}+\Delta \wt{M}^n_{t_i} + A_{t_i}^n\Delta t$ with
\begin{align}
	& \Delta \overline{M}^n_{t_i}   := \Delta M_{t_i}^n-\sigma_1 (X^{n}_{t_i})^\gamma  \Delta W_{t_i}, \label{barM}\\
	& \Delta \widehat{M}^n_{t_i} :=  \frac{\sigma_1^2 }{2}\frac{(X^{n}_{t_i})^{2\gamma-1}}{(1+k\Delta t)^2}((\Delta W_{t_i})^2-\Delta t ),\label{hatM}\\
&  \Delta \wt{M}^n_{t_i} := \frac{\Delta M^D_{t_i}}{4(1+k\Delta t)^2}\label{tildeM}\\
&   A_{t_i}^n \Delta t  :=  
(\Delta t)^2\Big[
		 -ak_n 
		-
	\frac{\sigma_1^2}{2}\frac{k (X^{n}_{t_i})^{2\gamma-1}  }{(1+k\Delta t)^2}\Big]
		+\frac{\E [D_{t_{i}}^- | \mathcal{F}_{t_i}]}{4(1+k\Delta t)^2},\label{At}
	\end{align}
	and $\Delta M_{t_i}^D := D_{t_{i}}^- - \E [D_{t_{i}}^- | \mathcal{F}_{t_i}]$. We point out that $\Delta \wt M^n$ and $A_{t_i}^n \Delta t$ have only moments up to but not including $\alpha$.
	Finally, by summing over $i= 1,\dots, n-1$ in \eqref{Scheme:expanded}, we extend
	the discrete time scheme in \eqref{Scheme:expanded} to continuous time, and write $X_t^{n} = \overline{X}_t^{n} + \overline R^n_t$ where
	\begin{align}
		&\overline{X}_t^{n} = x_0 + \int_0^t(a-k_n X_{\eta(s)}^{n})ds + \int_0^t\sigma_1(X_{\eta(s)}^{n})^\gamma dW_s + \int_0^t \sigma^n_2( X_{\eta(s)}^{n})^{\rho}dZ_s + \overline M^n_t + \widehat M^n_t , \nonumber\\
		&\overline R^n_t = \widetilde M_t^n + \int^t_0 A_{\eta(s)}^n ds, \label{continuous:extension}
	\end{align}
where for $t\in (t_i,t_{i+1}]$, $i=0,...,n-1$,
we set $\eta(t) := t_i$ and $\overline M_t^n := \E [\overline M_{t_{i+1}}^n  | \mathcal{F}_t]$. The continuous-time extensions of $\widehat M^n$ and $\widetilde M^n$ are similarly defined.

\brem\label{Remark:splitting:cont:ext}
The reason for introducing the quantities in \eqref{barM} - \eqref{At} and separating the scheme $X^n$ into $\overline X^n + \overline R^n$ will be apparent in Theorem \ref{Strong:Convergence}. Here, we mention only that, in Theorem \ref{Strong:Convergence}, we apply the Yamada-Watanabe approximation technique to $X - \overline X^n$ rather than $X- X^n$ and this allows us to exploit the fact that $\overline M$ and $\widehat M$ are both square integrable. The use of $\overline X^n$ is key in achieving a rate of convergence which improves upon the rate previously obtained in Li and Taguchi \cite{LiTaguchi, LiTaguchiEuler} and Frikha and Li \cite{FL}.
\erem

\subsection{Auxiliary lemmas and estimates}\label{Scheme:moment:estimates}
In this subsection, we compile a list of auxiliary results and estimates which are used in the proof of Theorem \ref{Strong:Convergence}. 
As the proofs of these auxiliary results are rather long, for presentation purposes, we have placed all of them in the appendix and for readers who are interested in the main result, please go directly to Theorem \ref{Strong:Convergence}. Here we only mention that Lemma \ref{probability:D:negative} and Lemma \ref{scheme:beta:moments} are key to the proof of convergence, and Assumption \ref{Assumption:3} is only used in Lemma \ref{CEV:positivity} and Lemma \ref{CEV:inverse:moments} which are later applied to remove the dependence of the convergence rate on $\gamma$ in Theorem \ref{Strong:Convergence}.

\begin{pro}[\cite{Sato}]
The compensated spectrally positive $\alpha$-stable process $Z$ is time-invariant (or selfsimilar, or strictly stable),
i.e. for $a>0$, $(Z_{at})_{t\geq 0}$ and $( a^{1/\alpha} Z_t)_{t\geq 0}$ are identical in law.
\end{pro}


\begin{pro}[\cite{JMS}]\label{prop:Laplace:alpha:stable}
For the compensated spectrally positive $\alpha$-stable process $Z$, for real number $s\geq 0$, the Laplace transform is given by
	\begin{equation*}
\E [e^{-sZ_t}] = \exp\Big( \frac{s^\alpha t}{\cos(\pi\alpha/2)}\Big).
	\end{equation*}
\end{pro}

\begin{lem}[Lemma A.1 \cite{FL}]\label{CEV:moments}
For $\beta\in[1,\alpha)$ we have $	\E [\,\sup_{t\leq T}  X_{t}^{\beta}] \leq  \infty$.
\end{lem}
By adopting a technique from Szpruch et al. \cite{SMHP} we show in Lemma \ref{CEV:positivity} that under Assumption \ref{Assumption:3} the jump-extended CEV process $X$ is strictly positive.  Then we obtain, in Lemma \ref{CEV:inverse:moments}, inverse moment estimates of $X$ on the whole time horizon $[0, T]$.

\begin{lem}\label{CEV:positivity}
Suppose Assumption \ref{Assumption:3} holds then $\P(X_t \in (0,\infty), \forall t > 0) = 1$.
\end{lem}




\begin{lem}\label{CEV:inverse:moments}
Suppose Assumption \ref{Assumption:3} holds then there exists a positive constant $C$ such that for $p>0$
	\[
	\sup_{t\leq T}\E[X_t^{-p}]\leq (x_0^{-p}+ CT) e^{kpT}.
	\]
	
\end{lem}


\begin{lem}\label{finiteE} 
The expected value of $X^n$ is finite or more specifically,
	\[\sup_{n}\max_{i=0,1,\dots,n} \mathbb{E} [X_{t_{i}}^{n}] < \infty.\]
	\end{lem}

\begin{lem}\label{2.3}
There exists a constant $C>0$ and a sufficiently large $p > 1$ such that for $\beta \in [1,\alpha)$
		\[\mathbb{E}[|D_{t_{i}}^-|^\beta]
		\leq
		C n^{-\frac{\beta}{\alpha}} \big(1 + \E[|X_{t_i}^{n}|^{\beta}]\big)
	\,\P(D_{t_{i}} < 0 \,)^\frac{1}{p}.
\]
\end{lem}

\noindent Next, we give two key estimates. The first is Lemma \ref{probability:D:negative} where we show the probability that the discriminant $D$ is negative is exponentially small with respect to the step size $\Delta t = Tn^{-1}$. The second is Lemma \ref{scheme:beta:moments} where we show that, for $\beta \in (1,\alpha)$, the $\beta$-moment of
our scheme is finite.

\begin{lem}\label{probability:D:negative}
There exist constants $C$, $C'>0$ such that 
\[\max_{i=0,1,...,n-1}\P\big(D_{t_{i}} < 0 \,\big) \,\leq\, C'\exp(-C n^{(\frac{1}{\alpha}+\rho-1)}).
\]
If $\sigma_2=0$ then we have $\P(D_{t_i}<0)=0$ for all $i=0,1,...,n-1$.

\end{lem}

\begin{lem}\label{scheme:beta:moments}
For $\beta\in(1,\alpha)$ we have $\sup_n \mathbb{E}\big[\max_{i=0,1,...,n} |X^n_{t_{i}}|^\beta\big] < \infty.$
\end{lem}
\noindent We now give estimates of the remainder term $R = \widehat{M}^n + \overline{M}^n+ \wt{M}^n + \int A_{\eta(s)}^n ds $ and $\overline R=\wt{M}^n + \int A_{\eta(s)}^n ds$.
\begin{lem}\label{martingale:components}
For $\beta \in[1,2]$, we have
\begin{align*}	
\max_{i=0,1,\dots,n-1} \E[|\Delta \widehat{M}^n_{t_i}|^\beta] \leq C n^{-\beta} 
\quad\mathrm{and}\quad
\max_{i=0,1,\dots,n-1} \E[|\Delta \overline{M}^n_{t_{i}}|^\beta] \leq C n^{-\frac{\beta}{2}-\frac{\beta}{2\alpha}},
\end{align*}	
and for any $t \in [0,T]$, we have $\E[|\widehat M^n_t|^{\beta}]\leq C n^{-\frac{\beta}{2}}$
and $\E[|\overline M^n_t|^{\beta}]\leq C n^{-\frac{\beta}{2\alpha}}$. In the case $\sigma_2 = 0$, we have 
\begin{gather*}
\E[|\Delta \overline{M}^n_{t_{i}}|^\beta] \leq C n^{-{\beta}} \quad \mathrm{and} \quad \E[|\overline M^n_t|^{\beta}]\leq C n^{-\frac{\beta}{2}}.
\end{gather*}
\end{lem}

\begin{lem}\label{alphamartingale}
For $\beta \in[1,\alpha)$, we have
\begin{align*}	
\max_{i=0,1,\dots,n-1} \E[|\Delta \widetilde M^n_{t_i}|^\beta] \leq  C n^{-2\beta}
\quad  \mathrm{and} \quad
\max_{i=0,1,\dots,n-1} \E[|A_{t_i}^n\Delta t|^{\beta}] \leq	C  n^{-2\beta},
\end{align*}	
and for any $t \in [0,T]$, we have $\E[|\widetilde M^n_t|^{\beta}]\leq C n^{-\beta}$
and $\E \big[\big| \int_0^t A_{\eta(s)}^n ds \big|^{\beta}\big]\leq C n^{-\beta}$.
\end{lem}

\bcor\label{cor1}
For $\beta \in [1,\alpha)$ we have $\E[|\overline R^n_{t}|^\beta] < C_Tn^{-\beta}$.
\ecor

\begin{lem}\label{LocalLocal}
For $\beta\in[1,\alpha)$ there exists a positive constant $C$ such that
	\[ 
	\sup_{t\leq T}\E[|\overline X^{n}_t - X^{n}_{\eta(t)}|^{\beta}]
	\leq  C n^{-\frac{\beta}{2}}.
	\]	
\end{lem}

\subsection{The Yamada-Watanabe approximation technique}
For completeness, we include the Yamada-Watanabe approximation technique
(see for example
Yamada and Watanabe \cite{YW},
Gy{\"o}ngy and R{\'a}sonyi \cite{GyongyRasonyi},
Li and Mytnik \cite{LiMytnik},
or Li and Taguchi \cite{LiTaguchiEuler}).
For each $\delta \in(1,\infty)$ and $\varepsilon\in(0,1)$ we select a continuous function $\psi_{\delta,\varepsilon} :  \R \rightarrow \R^+ $
with support $[\varepsilon/\delta, \varepsilon]$,
and the function $\psi_{\delta,\varepsilon}$ satisfies that
\[\int_{\varepsilon/\delta}^{\varepsilon}\psi_{\delta,\varepsilon} (z)dz = 1
\hspace{0.2cm} \text{and} \hspace{0.2cm}0\leq \psi_{\delta,\varepsilon} (z) \leq \frac{2}{z\log \delta},
\hspace{0.2cm}
\forall z>0.\]
Define a function $\phi_{\delta,\varepsilon} \in C^2(\R ;\R)$ by setting
\begin{align}
\phi_{\delta,\varepsilon} (x)
:=
\int_0^{|x|} \int_0^y
\psi_{\delta,\varepsilon} (z)dzdy. \label{YWfunction}
\end{align}
Then this $C^2$ function $\phi_{\delta,\varepsilon}$ satisfies the following useful properties:
\begin{align}
&|x| \leq \varepsilon + \phi_{\delta,\varepsilon}(x),
		\hspace{0.2cm}
		\text{for any}
		\hspace{0.2cm}
		x\in\R; \label{YW1}\\
&0\leq |\phi_{\delta,\varepsilon}'(x)|\leq 1,
		\hspace{0.2cm}
		\text{for any}
		\hspace{0.2cm}
		x\in\R; \label{YW2}\\
&\phi_{\delta,\varepsilon}'(x)\geq 0
		\hspace{0.2cm}
		\text{for}
		\hspace{0.2cm}
		x\geq 0;
		\hspace{0.2cm}
		\phi_{\delta,\varepsilon}'(x)\leq 0
		\hspace{0.2cm}
		\text{for}
		\hspace{0.2cm}
		x\leq 0; \nonumber \\
&\phi_{\delta,\varepsilon}''(\pm |x|)
		=
		\psi_{\delta,\varepsilon}(|x|)
		\leq
		\frac{2}{|x|\log \delta}\I_{[\varepsilon/\delta,\varepsilon]}(|x|)
		\leq
		\frac{2\delta}{\varepsilon\log \delta},
		\hspace{0.2cm}
		\text{for any}
		\hspace{0.2cm}
		x\in\R \setminus \{0\}. \label{YW4}
\end{align}
In addition, to estimate the jump terms, we include the following key lemmas from Li and Taguchi \cite{LiTaguchiEuler} and Frikha and Li \cite{FL}. For the reader's convenience, we have included the proofs in the appendix.
\begin{lem}[\cite{LiTaguchiEuler}, Lemma 1.3]\label{LTE1.3}
	Suppose the L{\'e}vy measure $\nu$ satisfies $\int_0^{\infty} (z\wedge z^2) \nu (dz)<\infty$.
	Let $\delta\in(1,\infty)$ and $\varepsilon\in(0,1)$.
	Then for any $x\in\R, y\in\R \setminus \{0\}$ with $xy\geq 0$
	and $u>0$, it holds that
\begin{align*}
&\int_0^{\infty}
	\{\phi_{\delta,\varepsilon}(y+xz)-\phi_{\delta,\varepsilon}(y) - xz\phi_{\delta,\varepsilon}'(y) \} \nu (dz)\\
&\hspace{2cm}\leq
	2\cdot\I_{(0,\varepsilon]}(|y|) 
	\left\{
	\frac{|x|^2}{\log \delta} \left(\frac{1}{|y|}\wedge \frac{\delta}{\varepsilon} \right)\int_0^u z^2\nu(dz)
	+ |x|\int_u^{\infty} z\nu(dz)
	\right\}.
\end{align*}
\end{lem}

\begin{lem}[\cite{FL}, Lemma 4.3]\label{xPrime}
	For $\alpha\in(1,2)$, let $\nu$ be the $\alpha$-stable L{\'e}vy measure given by
	$\nu(x) = x^{-\alpha-1}\I_{(0,\infty)}(x)dx$.
	Let $\delta\in(1,\infty)$ and $\varepsilon\in(0,1)$.
	Then for any $\alpha_0 \in[\alpha,2]$
	there exists a positive constant $C$
	such that for any $u\in(0,\infty)$
	and
	any $x,x',y\in\R$
	satisfying $x'y\geq 0$ we have
\begin{align}
&\int_0^{\infty}|\phi_{\delta,\varepsilon}(y+xz)-\phi_{\delta,\varepsilon}(y+x'z) - (x-x')z\phi_{\delta,\varepsilon}'(y) | \nu (dz) \nonumber\\
&\leq C\Big[ \Big(\frac{\delta}{\varepsilon \log \delta} +1\Big) |x-x'|^{\alpha_0} + |x-x'| \\
&\qquad +|x-x'| \Big\{ \frac{\I_{(0,\varepsilon]}(|y|)}{\log \delta} \Big( \frac{1}{|y|}\wedge \frac{\delta}{\varepsilon} \Big)|x'|\int_0^u z^2\nu(dz) + \int_u^{\infty} z\nu(dz)  \Big\} \Big]. \nonumber
\end{align}
\end{lem}

\brem 
The proof of Lemma \ref{xPrime} is identical to that of Lemma 4.3 of Frikha and Li \cite{FL} except that $\alpha_0$ is allowed to take value $\alpha$ here since the form of the L\'evy measure is known explicitly.
\erem

\subsection{Strong rate of convergence}\label{section:convergence}
In this subsection, we state our main result on the strong rate of convergence. We present only the result with Assumption \ref{Assumption:3}. This is because the rate of convergence without Assumption \ref{Assumption:3} (under only Assumption \ref{Assumption:1} and Assumption \ref{Assumption:2}) is far from optimal and can be easily deduced from the proof of Theorem \ref{Strong:Convergence}.

\begin{theorem}\label{Strong:Convergence}
Suppose Assumptions \ref{Assumption:3} holds then there exist some constant $C_T>0$ such that 
\begin{gather*}
	\sup_{t\leq T}\E[|X_t - X_t^{n}|] \leq C_Tn^{-\frac{1}{2}q(\alpha,\rho)} 
\end{gather*}
where $q(\alpha,\rho)$ is given by $\frac{\alpha_-}{2}\wedge \frac{1}{\alpha}\wedge \rho$ and the constant $\alpha_- \in(1,\alpha)$ can be chosen arbitrary close to $\alpha$.
\end{theorem}

\brem
Without going into technical details, the term $1/\alpha$ in $q(\alpha,\rho)$ derives from the increment of the $\alpha$-stable process in the martingale $\overline M^n$ defined through \eqref{barM}.
While the appearance of the term $\alpha_-/2$ is due to an integrability constraint in computing \eqref{EJ}, more specifically, the term $|\overline X^n-X^n_{\eta}|$ has only moments up to $\alpha$. Finally, the term $\rho$ stems from the H\"older exponent of the jump coefficient. In the case of the diffusion CEV, the jump coefficient is zero and hence the rate $1/\alpha$ does not appear in the estimate of the martingale $\overline M^n$, and the scheme $X^n$ is square integrable.
That is, one can take $\alpha_- = 2$ and both $1/\alpha$ and $\rho$ will not appear in $q(\alpha,\rho)$, hence resulting in a convergence rate of $1/2$. For more details see Corollary \ref{t1c}.
\erem

\brem
Although we are unable to provide a proof, our intuition suggest that the optimal rate of convergence is $\alpha_-/4$. This intuition comes from the fact that as $\alpha \uparrow 2$, we approach the Gaussian case in which we expect to have a convergence rate of $1/2$.
\erem

\begin{proof}
We recall that our scheme $X^n$ is given in \eqref{Scheme:expanded} and it can be decomposed into $X^n = \overline{X}^n + \overline{R}^n$  where
	\begin{align}
		&\overline{X}_t^{n} = x_0 + \int_0^t(a-k_n X_{\eta(s)}^{n})ds + \int_0^t\sigma_1(X_{\eta(s)}^{n})^\gamma dW_s + \int_0^t \sigma^n_2( X_{\eta(s)}^{n})^{\rho}dZ_s + \overline M^n_t + \widehat M^n_t , \nonumber\\
		&\overline R^n_t = \widetilde M_t^n + \int^t_0 A_{\eta(s)}^n ds.\nonumber 
	\end{align}
From the triangular inequality we have $|X - X^{n}| \leq |X-\overline{X}^n| + |\overline{R}^n|$ and by Corollary \ref{cor1} we have the estimate
	$\E[|\overline{R}^n_t|]\leq Cn^{-1}$. From the martingale representation theorem for square integrable martingales in a L\'evy filtration, see for example Theorem 5.3.6 in Applebaum \cite{Applebaum}, we know that there exist predictable and square-integrable processes $F$ and $G$ such that for $t\geq 0$ we have
	\[\overline{M}^n_t +\widehat{M}^n_t = \int_0^t F(s)dW_s + \int_0^t \int_0^\infty G(s,z)\widetilde{N}(dz,ds).\]
In view of the above, we focus on estimating $\E[|X-\overline{X}^n|]$. To proceed with the Yamada-Watanabe approximation, we set $\overline{Y}^n:=X-\overline{X}^n$ and denote its jumps by $\Delta \overline{Y}^n_t (z) := \sigma_2 [X_{t-}^{\rho} - (X_{\eta(t)}^{n})^{\rho}]z-G(t,z)$.
	Suppose $\varepsilon \in (0,1)$ and $\delta>1$,
	using property \eqref{YW1} and the It\^o formula, see for example Theorem 4.4.7 in \cite{Applebaum}, we have
	\begin{equation*}
	|\overline{Y}^{n}_t|
	\leq
	\varepsilon + \phi_{\delta,\varepsilon}(\overline{Y}^{n}_t)
	=
	\varepsilon
	+M_t^{n,\delta,\varepsilon}
	+I_t^{n,\delta,\varepsilon}
	+J_t^{n,\delta,\varepsilon}
	+K_t^{n,\delta,\varepsilon},
	\end{equation*}
	where $\phi_{\delta,\varepsilon}$ is the Yamada-Watanabe function given in \eqref{YWfunction} and the terms above are given by
\begin{align*}
M_t^{n, \delta,\varepsilon}
&:= \int_0^t \phi_{\delta,\varepsilon}' (\overline{Y}^{n}_{s-}) \{ \sigma_1[ X_{s-}^\gamma  - (X_{\eta(s)}^{n})^\gamma]-F(s) \}dW_s \nonumber \\
&\qquad + \int_0^t \int_0^{\infty} \left\{ \phi_{\delta,\varepsilon}(\overline{Y}^{n}_{s-} + \Delta \overline{Y}^{ n}_s (z) ) - \phi_{\delta,\varepsilon}(\overline{Y}^{ n}_{s-}) \right\} \widetilde{N}(ds,dz),\\
I_t^{n, \delta,\varepsilon} 
&:= \int_0^t \phi_{\delta,\varepsilon}'(\overline{Y}^{ n}_{s-}) \{ -k X_{s-} + k_n X_{\eta(s)}^{ n}\} ds,\\
J_t^{n, \delta,\varepsilon} 
&:= \frac{1}{2} \int_0^t \phi_{\delta,\varepsilon}'' (\overline{Y}^{ n}_{s-})\{ \sigma_1 [ X_{s-}^\gamma  - (X_{\eta(s)}^{ n})^\gamma]-F(s) \}^2 ds, \\
	K_t^{n, \delta,\varepsilon} 
&:=  \int_0^t \int_0^{\infty} \{\phi_{\delta,\varepsilon}(\overline{Y}^{ n}_{s-} + \Delta \overline{Y}^{n}_s (z) ) - \phi_{\delta,\varepsilon}(\overline{Y}^{ n}_{s-}) - \Delta \overline{Y}^{n}_s (z)  \phi_{\delta,\varepsilon}'(\overline{Y}^{ n}_{s-}) \} \nu(dz)ds.
\end{align*}
	Using the standard localisation arguments, the martingale term $M_t^{n,\delta,\varepsilon}$ can be eliminated by taking the expectation, 
and we focus on finding upper estimates for $I_t^{n,\delta,\varepsilon}$, $J_t^{n,\delta,\varepsilon}$ and $K_t^{n,\delta,\varepsilon}$.

\newpage

\noindent {\bf Estimates for the drift and diffusion term $I$ and $J$:} Let us first consider the term $I_t^{n,\delta,\varepsilon}$.
	Note that the quantity $|k-k_n|$ is of order $n^{-1}$
	and we can apply Lemma \ref{CEV:moments}, Lemma \ref{LocalLocal} and property \eqref{YW2}
	to obtain
	\begin{align}
		\E[|I_t^{n,\delta,\varepsilon}|] 
	&\leq
		\E \big[ \int_0^t |\phi_{\delta,\varepsilon}'(\overline{Y}^{n}_{s-})| |k_n X_s - k_n X_{\eta (s)}^{n}
		+ (k-k_n)X_s| ds \big] \nonumber \\
	&\leq  |k_n|\int_0^t \big(\E[|X_s -  X_{s}^{ n}|]  + \E [|X_s^{n} -  X_{\eta(s)}^{ n}|] \big)ds  + |k-k_n| \int_0^t \E[|X_s|]ds \nonumber \\
	&\leq C_T \Big(
		\int_0^t \E[|X_s -  X_{s}^{n}|] ds +  n^{-\frac{1}{2}}\Big). \label{eqI}
	\end{align}
Next we consider the term $J^{n,\delta,\varepsilon}$. By Jensen's inequality
\begin{align}
J_t^{n, \delta,\varepsilon} 
	&	= \frac{1}{2} \int_0^t \phi_{\delta,\varepsilon}'' (\overline{Y}^{n}_{s}) \left\{ \sigma_1 (X_{s}^\gamma- (X_{\eta(s)}^{n})^\gamma )-F(s) \right\}^2 ds \nonumber\\
	&	\leq C\int_0^t \phi_{\delta,\varepsilon}''(\overline{Y}^{ n}_{s}) \left(X_{s}^\gamma- (X_{\eta(s)}^{ n})^\gamma\right)^2 ds + C\int_0^t \phi_{\delta,\varepsilon}''(\overline{Y}^{ n}_{s})  F(s)^2 ds. \label{estimateJ}	
\end{align}
The second term in \eqref{estimateJ} can be estimated using \eqref{YW4} and Lemma \ref{martingale:components}, where we have $\E[\int_0^t F(s)^2  ds] \leq \E[|\overline M^n_t +  \widehat M^n_t|^2]  \leq C n^{-\frac{1}{\alpha}}$. To estimate the first term in \eqref{estimateJ}, we first recall, for $x, y \in \mathbb{R}_+$ and $b>a \geq 0$,
\begin{equation}
|x^a - y^a| \leq |x^b - y^b|x^{-(b- a)}. \label{bossy}
\end{equation}
Then for any $\alpha_-$ such that $2\gamma < \alpha_- < \alpha$  we have
\begin{align}
 & \int_0^t \phi_{\delta,\varepsilon}'' (\overline{Y}^{ n}_{s}) \left(X_{s}^\gamma- (X_{\eta(s)}^{ n})^\gamma\right)^2 ds 
\leq C\int_0^t \phi_{\delta,\varepsilon}'' (\overline{Y}^{ n}_{s}) |X_{s}^{\alpha_-/2} - (X^n_{\eta(s)})^{\alpha_-/2}|^2 X_{s}^{-(\alpha_- -2\gamma)}ds\nonumber \\
&\leq C\int_0^t \phi_{\delta,\varepsilon}'' (\overline{Y}^{ n}_{s}) |\overline{Y}^n_s|^{2} X_s^{-2(1-\gamma)} ds + C\int_0^t \phi_{\delta,\varepsilon}'' (\bar{Y}^{ n}_{s})  |\overline X_{s}^n- X^n_{\eta(s)}|^{\alpha_-}  X_{s}^{-(\alpha_- -2\gamma)}ds. \label{bossyexample}
\end{align}
To obtain the second inequality in the above, by noticing that $\overline X^n$ is not positive, we first apply the triangular inequality to write $|(X)^{\frac{\alpha_-}{2}} - (X^n_{\eta})^{\frac{\alpha_-}{2}}|  \leq |X^{\frac{\alpha_-}{2}} - ((\overline X^n)^+)^{\frac{\alpha_-}{2}}| + |((\overline X^n)^+)^{\frac{\alpha_-}{2}} - (X^n_\eta)^{\frac{\alpha_-}{2}}|.$
Then we apply inequality \eqref{bossy}, with $b= 1$ and $a = \alpha_-/2$, to the first term in the sum and conclude using the fact that the function $x\mapsto (x^+)^\zeta$ for $\zeta \in (0,1]$ is H\"older continuous with H\"older exponent $\zeta$.

By using the indicator function in \eqref{YW4} we deduce that $|\overline Y^n_s| \leq \epsilon$ and, from the inverse moment estimates in Lemma \ref{CEV:inverse:moments}, the expected value of the above can be further upper bounded by
\begin{gather*}
C \left\{\varepsilon  +  \frac{\delta}{\varepsilon \log\delta}  \mathbb{E}[\int_0^t|\overline X_{s}^n- X^n_{\eta(s)}|^{\alpha_-} \,  X_{s}^{-(\alpha_- -2\gamma)}ds]\right\},
\end{gather*}
with the constant $\alpha_- < \alpha$. To proceed, we select $p>1$, so that $p\alpha_-  < \alpha$, and let $q$ be the H{\"o}lder conjugate of $p$. Then by applying H\"older's inequality with $p$ and $q$, we obtain
\begin{align}
& \frac{C\delta}{\varepsilon\log \delta} \E[\int_0^t |\overline{X}^n_s- X_{\eta(s)}^{ n}|^{\alpha_-}\, X_s^{-(\alpha_- -2\gamma)}  ds] \label{holderexample}\\
&  \leq\frac{C\delta}{\varepsilon\log \delta}  \int_0^t \E[ |\overline{X}^n_s- X_{\eta(s)}^{ n}|^{p\alpha_-}]^{\frac{1}{p}} \E[X_s^{-q(\alpha_- -2\gamma)} ]^{\frac{1}{q}}ds
\leq \frac{C\delta n^{-\frac{\alpha_-}{2}}}{\varepsilon\log \delta}.\nonumber 
\end{align}
Note that although the H{\"o}lder conjugate $q$ can be very large, we can still control the inverse moments using Lemma \ref{CEV:inverse:moments}, and obtain an upper estimate independent of $p$ and $q$. More explicitly, we know from Lemma \ref{CEV:inverse:moments} that for $\gamma\in(1/2,1)$
	\[	
	\sup_{t\leq T} \E[|X_t|^{-q}]
	\leq (x_0^{-q}+ C_fT) \exp(qkT)
	\leq C_1 \exp\{C_2q\}.
	\]
By combining all the above estimates, we obtain
	\begin{gather}
	\E[|J_t^{n, \delta,\varepsilon}|] \leq C\left\{ \frac{\varepsilon}{\log \delta} + \frac{\delta}{\varepsilon \log \delta}(n^{-\frac{\alpha_{-}}{2}}+  n^{-\frac{1}{\alpha}}
	)\right\}. \label{EJ}
	\end{gather}

It is important to point out here that the worst rate in the remainder $R^n$ is coming from the square integrable martingale $\overline M^n$. This observation motivated our choice of the proxy process $\overline X^n$ and, by doing so, one can maximise the rate by taking advantage of square integrability.
Before proceeding to the estimate of the jump terms, we point out that the above estimates obtained in \eqref{eqI} and \eqref{EJ} are purely associated with the diffusion part and, although some later computations might not be optimal, we will not strife to improve them as long as the obtained rates are not worse than those obtained in \eqref{eqI} and \eqref{EJ}.

\vskip5pt
\noindent {\bf Estimates for the jump term $K$:} 
	Now, we estimate the term $K^{n,\delta,\varepsilon}_t$. Again we point out that, in the following, the method of proof is rather repetitive and we mainly make use of Lemma \ref{LTE1.3}, Lemma \ref{xPrime}, inequality \eqref{bossy} and, for $\alpha_0 \in [\alpha, 2]$, the quantities
	\begin{equation}
	I^{\alpha_0}_x :=x^{\alpha_0-2} \int_0^x z^2 \nu(dz) = \frac{x^{\alpha_0-\alpha}}{2-\alpha}
	\quad \text{and} \quad
	J^{\alpha_0}_x := x^{\alpha_0-1} \int_x^\infty z \nu(dz) = \frac{x^{\alpha_0-\alpha}}{\alpha-1} \label{IJ}
	\end{equation}
where we note that the constants in \eqref{IJ} diverges as $\alpha$ approaches one or two. To this end, we first decompose $K_t^{n,\delta,\varepsilon}$ into
	$K_t^{n,\delta,\varepsilon} = K_t^{n,\delta,\varepsilon,1}+K_t^{n,\delta,\varepsilon,2}+K_t^{n,\delta,\varepsilon,3}$ where
\begin{align*}
K_t^{n,\delta,\varepsilon,1}
& := \int_0^t \int_0^{\infty} \{\phi_{\delta,\varepsilon}(\overline{Y}^{ n}_{s-}
    + \sigma_2^n[ X_{s-}^{\rho}  - ((\overline{X}_{s}^{n})^+)^{\rho}]z ) 
    - \phi_{\delta,\varepsilon}(\overline{Y}^{ n}_{s-})\\
&\quad- \sigma_2^n[ X_{s-}^{\rho}  - ((\overline{X}_{s}^{n})^+)^{\rho}]z  \phi_{\delta,\varepsilon}'(\overline{Y}^{ n}_{s-}) \} \nu(dz)ds,\\
K_t^{n,\delta,\varepsilon,2}
& := \int_0^t \int_0^{\infty} \{\phi_{\delta,\varepsilon}(\overline{Y}^{ n}_{s-} + \sigma_2^n[X_{s-}^{\rho} - (X_{\eta(s)}^{n})^{\rho} ]z )\\
&\quad-\phi_{\delta,\varepsilon}(\overline{Y}^{ n}_{s-}+ \sigma_2^n[ X_{s-}^{\rho}  - ((\overline{X}_{s}^{n})^+)^{\rho}]z )- \sigma_2^n  [((\overline{X}_{s}^{n})^+)^{\rho} - (X_{\eta(s)}^{n})^{\rho} ]z  \phi_{\delta,\varepsilon}'(\overline{Y}^{n}_{s}) \} \nu(dz)ds,\\
K_t^{n,\delta,\varepsilon,3}
& := \int_0^t \int_0^{\infty} \{\phi_{\delta,\varepsilon}(\overline{Y}^{ n}_{s-} + \sigma_2^n[X_{s-}^{\rho} - (X_{\eta(s)}^{n})^{\rho} ]z - G(s,z)) \\
&\quad-\phi_{\delta,\varepsilon}(\overline{Y}^{ n}_{s-}
+ \sigma^n_2[X_{s-}^{\rho} - (X_{\eta(s)}^{n})^{\rho} ]z )+G(s,z)  \phi_{\delta,\varepsilon}'(\overline{Y}^{n}_{s}) \} \nu(dz)ds.
\end{align*}
To estimate the term $K_t^{n,\delta,\varepsilon,1}$, we set $y = \overline{Y}^n_{s-} = X_{s-} - \overline{X}^{n}_s$	and
$x = \sigma_2^n[ X_{s-}^{\rho}  - ((\overline{X}_{s}^{n})^+)^{\rho}]$.
Since $xy\geq 0$ we can apply Lemma \ref{LTE1.3} and inequality \eqref{bossy}, with $a = \rho$ and $b = 1$, to the term $|X_{s-}^{\rho} - ((\overline{X}^n_s)^+ )^{\rho} |$ to obtain for any $u>0$ the following upper estimate for $|K_t^{n,\delta,\varepsilon,1}|$,
\begin{align*}
& 	C\I_{(0,\varepsilon]}(|\overline Y^n_{s-}|) 
	\left\{
	\frac{|X_{s-}^{\rho}  - ((\overline{X}_{s}^{n})^+)^{\rho}|^2}{\log \delta} \left(\frac{1}{|\overline Y^n_{s-}|}\wedge \frac{\delta}{\varepsilon} \right)\int_0^u z^2\nu(dz)
	+ |X_{s-}^{\rho}  - ((\overline{X}_{s}^{n})^+)^{\rho}|\int_u^{\infty} z\nu(dz)
	\right\}\\
	& \leq 	C\I_{(0,\varepsilon]}(|\overline Y^n_{s-}|) 
	\left\{
	\frac{|X_{s-}  - (\overline{X}_{s}^{n})^+|^2}{X_{s-}^{2(1-\rho)}\log \delta} \left(\frac{1}{|\overline Y^n_{s-}|}\wedge \frac{\delta}{\varepsilon} \right)\int_0^u z^2\nu(dz)
	+ X_{s-}^{-(1-\rho)}|X_{s-}- (\overline{X}_{s}^{n})^+|\int_u^{\infty} z\nu(dz)
	\right\}\\
& \leq 	C\I_{(0,\varepsilon]}(|\overline Y^n_{s-}|) 
	\left\{
	\frac{|\bar Y^n_{s-}|^2}{X_{s-}^{2(1-\rho)}\log \delta} \left(\frac{1}{|\overline Y^n_{s-}|}\wedge \frac{\delta}{\varepsilon} \right)\int_0^u z^2\nu(dz)
	+ X_{s-}^{-(1-\rho)}|\overline Y^n_{s-}|\int_u^{\infty} z\nu(dz)
	\right\}
\end{align*}
where in the last inequality, we have used the fact that $x\mapsto x^+$ is H\"older continuous with H\"older exponent one. Then by using the indicator $\I_{(0,\varepsilon]}(|\overline Y^n_{s-}|)$ and \eqref{IJ}, we see that
\begin{align*}
|K_t^{n,\delta,\varepsilon,1}| &\leq \frac{C}{\log \delta} X_s^{-2(1-\rho)} \varepsilon \int_0^u z^2 \nu(dz) + C X_s^{-(1-\rho)} \varepsilon \int_u^{\infty} z \nu(dz)\\
&\leq C_T \varepsilon \Big( X_s^{-2(1-\rho)} \frac{1}{\log\delta}u^{2-\alpha_1} I^{\alpha_1}_{u} + X_s^{-(1-\rho)} u^{-(\alpha_1-1)}J^{\alpha_1}_u  \Big)
\end{align*}
where $\alpha_1 \in [\alpha, 2]$.
	Finally, by setting $u=\log(\delta)$ we see that the above can be upper estimated by
	\begin{align*}
	 C_T\, \varepsilon \log(\delta)^{1-\alpha}\Big( X_s^{-2(1-\rho)} I^{\alpha_1}_{\log(\delta)} + X_s^{-(1-\rho)} J^{\alpha_1}_{\log(\delta)}  \Big)
	\end{align*}
and by using Lemma \ref{CEV:inverse:moments}, we obtain
\begin{align}
\E[|K_t^{n,\delta,\varepsilon,1}|] \leq C_T\, \varepsilon \log(\delta)^{1-\alpha}\big( I^{\alpha_1}_{\log(\delta)} + J^{\alpha_1}_{\log(\delta)}  \big).\label{eq35}
\end{align}

To estimate the term $K_t^{n,\delta,\varepsilon,2}$,
we let $y= \overline{Y}^{ n}_{s-} $,
$x=\sigma^n_2\{X_{s-}^{\rho} - (X_{\eta(s)}^{n})^{\rho} \}$ and
$x' = \sigma_2^n\{X_{s-}^{\rho} - ((\overline{X}_{s}^{n})^+)^{\rho} \}$
in Lemma \ref{xPrime}. Since $x'y\geq 0$, we have for $\alpha_2 \in [\alpha,2]$:
\begin{align*}
& |K_t^{n,\delta,\varepsilon,2}| 
	 \leq C\Bigg[ \Big(\frac{\delta}{\varepsilon \log \delta} +1\Big) |((\overline{X}^n_s)^+ )^{\rho} - (X_{\eta(s)}^{n})^{\rho}|^{\alpha_2} + |((\overline{X}_{s}^{n})^+)^{\rho} - (X_{\eta(s)}^{n})^{\rho}| \\
	&+ |((\overline{X}^n_s)^+ )^{\rho} - (X_{\eta(s)}^{n})^{\rho}|\,\Big\{ \frac{\I_{(0,\varepsilon]}(|\overline{Y}_{s-}^n|)}{\log \delta} \Big( \frac{1}{|\overline{Y}_{s-}^n|}\wedge \frac{\delta}{\varepsilon} \Big)
|X_{s-}^{\rho} - ((\overline{X}^n_s)^+ )^{\rho} |\int_0^u z^2\nu(dz) 
+ \int_u^{\infty} z\nu(dz)  \Big\} \Bigg].
\end{align*}
By applying inequality \eqref{bossy}, with $b= \rho$ and $a = 1$, to the term $|X_{s-}^{\rho} - ((\overline{X}^n_s)^+ )^{\rho} |$, the right hand side above can be estimated by 
\begin{align*}
& C\Big[ \Big(\frac{\delta}{\varepsilon \log \delta} +1\Big) |\overline{X}^n_s - X_{\eta(s)}^{n}|^{\alpha_2\rho} 
		+ |\overline{X}^n_s - X_{\eta(s)}^{n}|^{\rho} + |\overline{X}^n_s - X_{\eta(s)}^{n}|^{\rho}
		\Big\{ \frac{\delta}{\log \delta} X_s^{-(1-\rho)} \int_0^u z^2\nu(dz) 
		+ \int_u^{\infty} z\nu(dz)  \Big\} \Big]\\
&\leq C\Big[ \Big(\frac{\delta}{\varepsilon \log \delta} +1\Big) |\overline{X}^n_s  - X_{\eta(s)}^{n}|^{\alpha_2\rho}
		+ |\overline{X}^n_s  - X_{\eta(s)}^{n}|^{\rho}		\\
& \qquad + |\overline{X}^n_s - X_{\eta(s)}^{n}|^{\rho} 
		  \Big\{ \frac{\delta}{\log \delta} u^{2-\alpha_3} I^{\alpha_3}_u + u^{-(\alpha_3-1)}J^{\alpha_3}_u  \Big\}\big(X_s^{-(1-\rho)}  +1 \big)	\Big]
\end{align*}
where $\alpha_3\geq \alpha$ is chosen later. We choose $u = \delta^{-1}\log(\delta)$, and obtain that
	\[
	\Big|\frac{\delta}{\log \delta} u^{2-\alpha_3} I^{\alpha_3}_u + u^{-(\alpha_3-1)}J^{\alpha_3}_u \Big|
	\leq
	C \Big(\frac{\delta}{\log\delta} \Big)^{\alpha_3-1} 
	\]
where we note that $\log(\delta)/\delta <1$ for $\delta>1$. 
Then by applying Lemma \ref{LocalLocal} and H\"older's inequality as done in \eqref{holderexample}, we obtain
	\begin{align}
		\E[|K_t^{n,\delta,\varepsilon,2}|] 
	&\leq
		C_T\Big[\Big(\frac{\delta}{\varepsilon \log \delta} +1\Big)n^{-\frac{\alpha_2\rho}{2}}
		+ n^{-\frac{\rho}{2}}
		+\Big(\frac{\delta}{\log\delta} \Big)^{\alpha_3-1} n^{-\frac{\rho}{2}}
		\Big]\nonumber \\
	&\leq
		C_T\Big[\frac{\delta}{\varepsilon \log \delta}\, n^{-\frac{\alpha_2\rho}{2}}
		+\Big(\Big(\frac{\delta}{\log\delta} \Big)^{\alpha_3-1}  + 1 \Big) n^{-\frac{\rho}{2}}
		\Big],	\label{eq36}
	\end{align}
	with the constraints $\alpha_2 \in [\alpha,2]\cap (0,\alpha/\rho)$ and $\alpha_3\in [\alpha,2]$.

To estimate the term $K^{n,\delta,\varepsilon,3}_t$,
	we further decompose it into 
	$K^{n,\delta,\varepsilon,3}_t = K^{n,\delta,\varepsilon,3,1}_t+K^{n,\delta,\varepsilon,3,2}_t + K^{n,\delta,\varepsilon,3,3}_t$ which are given by
\begin{align*}
K^{n,\delta,\varepsilon,3,1}_t
& 		:= \int_0^t \int_0^{\infty} \{\phi_{\delta,\varepsilon}(\overline{Y}^{ n}_{s-} + \sigma_2^n[X_{s-}^{\rho} - (X_{\eta(s)}^{n})^{\rho} ]z - G(s,z)) \\
&\quad  -\phi_{\delta,\varepsilon}(\overline{Y}^{ n}_{s-}
		+ \sigma_2^n[X_{s-}^{\rho} - (X_{\eta(s)}^{n})^{\rho} ]z )+G(s,z)  \phi_{\delta,\varepsilon}'(\overline{Y}^{ n}_{s-} + \sigma_2^n[X_{s-}^{\rho} - (X_{\eta(s)}^{n})^{\rho} ]z) \} \nu(dz)ds,\\
K^{n,\delta,\varepsilon,3,2}_t
&		:= \int_0^t \int_0^{\infty} G(s,z) \{\phi_{\delta,\varepsilon}'(\overline{Y}^{ n}_{s-}) -\phi_{\delta,\varepsilon}'(\overline{Y}^{ n}_{s-}+ \sigma_2^n[ X_{s-}^{\rho}  - ((\overline{X}^n_s)^+ )^{\rho}]z ) \}  \nu(dz)ds,\\
K^{n,\delta,\varepsilon,3,3}_t
&		:= \int_0^t \int_0^{\infty} G(s,z) \Big\{\phi_{\delta,\varepsilon}'(\overline{Y}^{ n}_{s-}+ \sigma_2^n[ X_{s-}^{\rho}  - ((\overline{X}^n_s)^+ )^{\rho}]z )\\
&\quad -\phi_{\delta,\varepsilon}'(\overline{Y}^{ n}_{s-} + \sigma_2^n[X_{s-}^{\rho} - (X_{\eta(s)}^{n})^{\rho} ]z)\Big\} \nu(dz)ds.
\end{align*}
By using the second-order Taylor expansion, the first term $K^{n,\delta,\varepsilon,3,1}_t$ can be estimated by
	\begin{align*}
	G(s,z)^2 \int_0^1 (1-\theta) \phi_{\delta,\varepsilon}'' \big(\overline{Y}^{ n}_{s-} + \sigma_2^n[X_{s-}^{\rho} - (X_{\eta(s)}^{n})^{\rho} ]z -\theta G(s,z)\big) \, d\theta 
	& \leq \frac{\delta G(s,z)^2}{\varepsilon \log \delta},
	\end{align*}
 and by Lemma \ref{martingale:components}, we obtain
\begin{align}
\E[|K^{n,\delta,\varepsilon,3,1}_t|]
&\leq \frac{\delta}{\varepsilon\log\delta}  \E[\int_0^t\int_0^\infty G(s,z)^2 \nu(dz) ds] \leq \frac{\delta}{\varepsilon\log\delta} 
\E[|\overline{M}^n_t+\widehat{M}^n_t|^2]\leq C \frac{\delta}{\varepsilon\log\delta}  n^{-\frac{1}{\alpha}}. \label{eq37}
\end{align}

To estimate the second term $K^{n,\delta,\varepsilon,3,2}_t$, we consider separately $\{z \leq u \}$ and $\{z > u \}$ for $u > 0$.
	On the set $\{z\leq u\}$ we apply the first-order Taylor expansion, the fact that $\overline{Y}^{ n}_{s-}[X_{s-}^{\rho}  - ((\overline{X}^n_s)^+)^{\rho}]\geq 0$ and inequality \eqref{bossy}, with $a= \rho$ and $b = 1$, to obtain	
\begin{align}
\E[|K_t^{n,\delta,\varepsilon,3,2}|]
	&\leq \E\big[\int_0^t \int_0^u  |G(s,z)| \I_{(0, \varepsilon]}(|\overline{Y}^{ n}_{s-}|)\, \frac{2\sigma_2^n}{\log \delta} \Big( \frac{1}{|\overline{Y}^{ n}_{s-}|} \wedge \frac{\delta}{\varepsilon} \Big) \,  z|X_{s-}^{\rho}  - ((\overline{X}^n_s)^+)^{\rho}|  \nu(dz)ds \big]\nonumber\\
&\leq \E\big[\int_0^t \int_0^u   \I_{(0, \varepsilon]}(|\overline{Y}^{ n}_{s-}|)\, \frac{2\sigma_2^n}{\log \delta} \Big( \frac{1}{|\overline{Y}^{ n}_{s-}|} \wedge \frac{\delta}{\varepsilon} \Big) \times \big( G(s,z)^2 +  z^2 X^{-2(1-\rho)} |X_{s-}- \overline{X}^n_s|^2 \big)  \nu(dz)ds \big]\nonumber\\
    &\leq C_T \frac{\delta}{\varepsilon\log\delta}\, n^{-\frac{1}{\alpha}} + C_T \frac{\varepsilon }{\log\delta} \int_0^u z^2 \nu(dz), \label{BzIndi}
\end{align}
	where we've used Young's inequality in the second last line. While on the set $\{z>u\}$, by applying H{\"o}lder's inequality we obtain 
	\begin{align}
	& \E\Big[\Big| \int_0^t \int_u^\infty G(s,z) \Big\{\phi_{\delta,\varepsilon}'(\overline{Y}^{ n}_{s-}) -\phi_{\delta,\varepsilon}'(\overline{Y}^{ n}_{s-}+ \sigma_2^n[ X_{s-}^{\rho}  - ((\overline{X}^n_s)^+ )^{\rho}]z ) \Big\}  \nu(dz)ds \Big|\Big] \nonumber \\
	&\leq \E\Big[\int_0^t \int_u^\infty G(s,z)^2 \nu(dz)ds \Big]^\frac{1}{2}\times \E\Big[\int_0^t \int_u^\infty  |\phi_{\delta,\varepsilon}'(\overline{Y}^{ n}_{s-}) -\phi_{\delta,\varepsilon}'(\overline{Y}^{ n}_{s-}+ \sigma_2^n[ X_{s-}^{\rho}  - ((\overline{X}^n_s)^+ )^{\rho}]z ) |^2 \nu(dz)ds \Big]^\frac{1}{2}. \nonumber
	\end{align}
To estimate the second term, we first use the property that $|\phi'_{\delta,\epsilon}| \leq 1$ in \eqref{YW2} to reduce the power of the integrand from two to one and then apply the first-order Taylor expansion with \eqref{YW4} and inequality \eqref{bossy}, with $a = \rho$ and $b= 1$, to obtain
	\begin{align}
& |\phi_{\delta,\varepsilon}'(\overline{Y}^{ n}_{s-}) -\phi_{\delta,\varepsilon}'(\overline{Y}^{ n}_{s-}+ \sigma_2^n[ X_{s-}^{\rho}  - ((\overline{X}^n_s)^+)^{\rho}]z ) |^2  \\
&  \leq  2 |\phi_{\delta,\varepsilon}'(\overline{Y}^{ n}_{s-}) -\phi_{\delta,\varepsilon}'(\overline{Y}^{ n}_{s-}+ \sigma_2^n[ X_{s-}^{\rho}  - ((\overline{X}^n_s)^+)^{\rho}]z ) | \nonumber\\
& \leq \I_{(0, \varepsilon]}(|\overline{Y}^{ n}_{s-}|)\, \frac{4\sigma_2^n }{\log \delta} \Big( \frac{1}{|\overline{Y}^{ n}_{s-}|} \wedge \frac{\delta}{\varepsilon} \Big) z  |\overline{Y}^{ n}_{s-}|X_{s-}^{-(1-\rho)}  \leq \frac{4\sigma_2^n}{\log \delta}z  X_{s}^{-(1-\rho)}.\nonumber 
	\end{align}
More explicitly, the last two inequalities can be obtained through the following estimates given by
	\begin{align}
&|\phi_{\delta,\varepsilon}'(\overline{Y}^{ n}_{s-}) -\phi_{\delta,\varepsilon}'(\overline{Y}^{ n}_{s-}+ \sigma_2^n[ X_{s-}^{\rho}  - ((\overline{X}^n_s)^+)^{\rho}]z ) |\nonumber \\
&\leq \int_0^1 \sigma_2^n | X_{s-}^{\rho}  - ((\overline{X}^n_s)^+)^{\rho}| z\, \phi_{\delta,\varepsilon}''(\overline{Y}^{ n}_{s-}+ \theta\sigma_2 |X_{s-}^{\rho}  - ((\overline{X}^n_s)^+)^{\rho} | z )\,(1-\theta)d\theta\nonumber \\
&\leq \int_0^1 \sigma_2^n |X_{s-}^{\rho}  - ((\overline{X}^n_s)^+)^{\rho}|z(1-\theta)\, \frac{2\I_{[\varepsilon/\delta,\varepsilon]}(|\overline{Y}^{ n}_{s-}+ \sigma_2^n[ X_{s-}^{\rho}  - ((\overline{X}^n_s)^+)^{\rho}]z|)}{|\overline{Y}^{ n}_{s-}+ \sigma_2^n[ X_{s-}^{\rho}  - ((\overline{X}^n_s)^+)^{\rho}z|\log \delta}\, d\theta. \label{TaylorOut}
	\end{align}
Since $\overline{Y}^{ n}_{s-}[X_{s-}^{\rho}  - ((\overline{X}^n_s)^+)^{\rho}]\geq 0$,
we can deduce that $|\overline{Y}^{ n}_{s-}+ \sigma_2^n[ X_{s-}^{\rho}  - ((\overline{X}^n_s)^+)^{\rho}]z|\geq |\overline{Y}^{ n}_{s-}|$
and hence
\[\I_{[\varepsilon/\delta,\varepsilon]}(|\overline{Y}^{ n}_{s-}+ \sigma_2^n[ X_{s-}^{\rho}  - ((\overline{X}^n_s)^+)^{\rho}]z|)
\leq
\I_{(0,\varepsilon]}(|\overline{Y}^{ n}_{s-}+ \sigma_2^n[ X_{s-}^{\rho}  - ((\overline{X}^n_s)^+)^{\rho}]z|)
\leq
\I_{(0,\varepsilon]}(|\overline{Y}^{ n}_{s-}|).\]
Thus for the fraction in \eqref{TaylorOut} we have
\begin{align*}
\frac{\I_{[\varepsilon/\delta,\varepsilon]}(|\overline{Y}^{ n}_{s-}+ \sigma_2^n[ X_{s-}^{\rho}  - ((\overline{X}^n_s)^+)^{\rho}]z|)}{|\overline{Y}^{ n}_{s-}+ \sigma_2^n[ X_{s-}^{\rho}  - ((\overline{X}^n_s)^+)^{\rho}]z|}
&\leq \I_{(0, \varepsilon]}(|\overline{Y}^{ n}_{s-}|) \Big( \frac{1}{|\overline{Y}^{ n}_{s-}|} \wedge \frac{\delta}{\varepsilon} \Big).
\end{align*}
Then by using inequality \eqref{bossy}, with $b = 1$ and $a = \rho$, we obtain $$|X_{s-}^{\rho}  - ((\overline{X}^n_s)^+)^{\rho}|\leq |X_{s-}  - (\overline{X}^n_s)^+|X_{s-}^{-(1-\rho)} \leq |\overline{Y}^{ n}_{s-}|X_{s-}^{-(1-\rho)}$$ which shows that \eqref{TaylorOut} can be estimated by
\begin{align*}
&\int_0^1 (1-\theta)d\theta\  \sigma_2^n[ X_{s-}^{\rho}  - (\overline{X}^n_s)^{\rho}]z\, \frac{2}{\log \delta}\I_{(0, \varepsilon]}(|\overline{Y}^{ n}_{s-}|) \Big( \frac{1}{|\overline{Y}^{ n}_{s-}|} \wedge \frac{\delta}{\varepsilon} \Big)\\
&\leq
\sigma_2^n\overline{Y}^{ n}_{s-}X_{s-}^{-(1-\rho)}z\, \frac{1}{\log \delta}\I_{(0, \varepsilon]}(|\overline{Y}^{ n}_{s-}|) \Big( \frac{1}{|\overline{Y}^{ n}_{s-}|} \wedge \frac{\delta}{\varepsilon} \Big)\leq \frac{\sigma_2^n}{\log\delta}z X_{s-}^{-(1-\rho)}.
\end{align*}
By Lemma \ref{CEV:inverse:moments} and Lemma \ref{martingale:components}, we have
	\begin{align}
		\E[|K_t^{n,\delta,\varepsilon,3,2}|]
		&\leq \E \Big[ \int_0^t \int_u^\infty G(s,z)^2   \nu(dz)ds \Big]^\frac{1}{2}
		\E \Big[ \frac{4\sigma_2^n\, }{\log \delta} \int_u^\infty  z  \nu(dz) \int_0^t X_{s}^{-(1-\frac{1}{\alpha})} ds\Big]^\frac{1}{2} \nonumber \\
	&\leq  C_T \Big[\frac{1}{\log \delta}\int_u^\infty  z  \nu(dz) \Big]^\frac{1}{2}	n^{-\frac{1}{2\alpha}}. \label{eq42}
	\end{align}
	Therefore by combining \eqref{BzIndi} and \eqref{eq42}, we have for $\alpha_4 \in [\alpha,2]$
\begin{align*}
\E[|K_t^{n,\delta,\varepsilon,3,2}|] & \leq C \Big( \frac{\delta}{\varepsilon\log\delta}\, n^{-\frac{1}{\alpha}} +  \frac{\varepsilon}{\log\delta} \int_0^u z^2 \nu(dz)+ \Big(\frac{1}{\log \delta}\int_u^\infty  z  \nu(dz) \Big)^\frac{1}{2}	n^{-\frac{1}{2\alpha}}\Big)\\
&\quad\leq C \Big( \frac{\delta}{\varepsilon\log\delta}\, n^{-\frac{1}{\alpha}}
		+ \big(I^{\alpha_4}_u+ (J^{\alpha_4}_u)^{\frac{1}{2}}\big) \big( \frac{\varepsilon}{\log\delta} u^{2-\alpha_4}  
		+ \frac{ 1  }{(\log \delta)^\frac{1}{2}} u^{-\frac{\alpha_4-1}{2}}n^{-\frac{1}{2\alpha}}\big)\Big).
\end{align*}
where $I^{\alpha_4}$ and $J^{\alpha_4}$ are given in \eqref{IJ}.	Then, by choosing $u= (n^{-\frac{1}{\alpha}}\, \varepsilon^{-2} \log\delta \,)^{\frac{1}{3-\alpha_4}}$,
	we balance the two quantities related to $u$ to obtain
	\begin{gather}
	\E[|K_t^{n,\delta,\varepsilon,3,2}|] \leq  C \Big( \frac{\delta}{\varepsilon\log\delta}\, n^{-\frac{1}{\alpha}} 
	+\frac{\varepsilon}{\log\delta}\big(n^{-\frac{1}{\alpha}}\, \varepsilon^{-2} \log\delta \, \big)^{\frac{2-\alpha_4}{3-\alpha_4}}\Big). \label{eq40}
	\end{gather}

To estimate the term $K_t^{n,\delta,\varepsilon,3,3}$, we proceed similarly to the proof of Lemma \ref{xPrime} by working separately on the sets $\{|\overline{X}^n_s- X_{\eta(s)}^{n}|>1\}$ and $\{|\overline{X}^n_s- X_{\eta(s)}^{n}|\leq 1\}$. 

To this end, we first work on the set $\{|\overline{X}^n_s- X_{\eta(s)}^{n}|\leq 1\}$ and suppose that $z > 1$. By applying H{\"o}lder's inequality, the first-order Taylor expansion and Young's inequality we obtain
\begin{align*}
	&\int_0^t \int_1^{\infty} G(s,z) \Big\{\phi_{\delta,\varepsilon}'(\overline{Y}^{ n}_{s-}+ \sigma_2^n[ X_{s-}^{\rho}  - ((\overline{X}^n_s)^+)^{\rho}]z ) -\phi_{\delta,\varepsilon}'(\overline{Y}^{ n}_{s-} + \sigma_2^n[X_{s-}^{\rho} - (X_{\eta(s)}^{n})^{\rho} ]z)\Big\} \nu(dz)ds\\
	&\leq C\Big( \int_0^t \int_1^\infty G(s,z)^2  \nu(dz)ds \Big)^\frac{1}{2}  \Big( \int_0^t \int_1^\infty \frac{\delta}{\varepsilon\log\delta}\, z |((\overline{X}^n_s)^+ )^{\rho} - (X_{\eta(s)}^{n})^{\rho}| \nu(dz)ds \Big)^\frac{1}{2}\\
	&\leq C\Big( \frac{\delta}{\varepsilon\log\delta}\int_0^t \int_1^\infty G(s,z)^2  \nu(dz)ds + \int_0^t \int_1^\infty  z |((\overline{X}^n_s)^+ )^{\rho} - (X_{\eta(s)}^{n})^{\rho}| \nu(dz)ds \Big)\\
	&= C\Big( \frac{\delta}{\varepsilon\log\delta}\int_0^t \int_0^\infty  G(s,z)^2 \nu(dz)ds + \int_0^t |\overline{X}^n_s- X_{\eta(s)}^{n}|^{\rho} ds \Big).
\end{align*}
On the other hand, for $z \leq 1$, by the first-order Taylor expansion, property \eqref{YW4} and Young's inequality, we observe that
	\begin{align*}
& \int_0^t \int_0^1 G(s,z)\Big\{\phi_{\delta,\varepsilon}'(\overline{Y}^{ n}_{s-}+ \sigma_2^n[ X_{s-}^{\rho}  - ((\overline{X}^n_s)^+)^{\rho}]z ) -\phi_{\delta,\varepsilon}'(\overline{Y}^{ n}_{s-} + \sigma_2^n[X_{s-}^{\rho} - (X_{\eta(s)}^{n})^{\rho} ]z)\Big\} \nu(dz)ds\\
&\leq \int_0^t \int_0^1 | G(s,z)| \, \frac{2\sigma_2^n \delta}{\varepsilon\log\delta}\, z |((\overline{X}^n_s)^+ )^{\rho}- (X_{\eta(s)}^{n})^{\rho} | \nu(dz)ds\\
&\leq C \int_0^t \int_0^1   \frac{\delta}{\varepsilon\log\delta} \big( G(s,z)^2 +z^2 |((\overline{X}^n_s)^+ )^{\rho}- (X_{\eta(s)}^{n})^{\rho} |^2 \big) \nu(dz)ds\\
&\leq C \Big( \frac{\delta}{\varepsilon\log\delta} \int_0^t \int_0^\infty    G(s,z)^2  \nu(dz)ds  + \frac{\delta}{\varepsilon\log\delta} \int_0^t |\overline{X}^n_s- X_{\eta(s)}^{n}|^{(2\rho )\wedge \alpha_-  } ds  \Big).
\end{align*}
where in the last inequality, the fact that  $|\overline{X}^n_s- X_{\eta(s)}^{n}| \leq 1$ is used to reduce the power of $|\overline{X}^n_s- X_{\eta(s)}^{n}|$ to $\alpha_-$ if $2\rho \geq \alpha$. Finally by taking the expected value of the estimates on $\{z \leq 1\}$ and $\{z > 1\}$, we obtain from Lemma \ref{LocalLocal} and Lemma \ref{martingale:components} that
\begin{gather}	
\E[|K_t^{n,\delta,\varepsilon,3,3}|] \leq C_T\Big( \frac{\delta}{\varepsilon\log\delta} \big( n^{-\frac{1}{\alpha}} + n^{-(\rho\wedge \frac{\alpha_-}{2})} \big) + n^{-\frac{\rho}{2}}\Big).\label{K331}
\end{gather}
\newpage
On the set $\{|\overline{X}^n_s- X_{\eta(s)}^{n}|>1\}$, we again consider separately upper estimates for $z\geq u$ and $z< u$ where we now set $u = |\overline{X}^n_s- X_{\eta(s)}^{n}|^{-\frac{1}{\alpha}} < 1$. For $z \geq u$, by Young's inequality, property \eqref{YW2} and the first-order Taylor expansion we have
\begin{align*}
&\int_0^t \int_u^{\infty} G(s,z) \Big\{\phi_{\delta,\varepsilon}'(\overline{Y}^{ n}_{s-}+ \sigma_2^n[ X_{s-}^{\rho}  - (\overline{X}^n_s)^{\rho}]z ) -\phi_{\delta,\varepsilon}'(\overline{Y}^{ n}_{s-} + \sigma_2^n[X_{s-}^{\rho} - (X_{\eta(s)}^{n})^{\rho} ]z)\Big\} \nu(dz)ds\\
&\leq C\Big( \frac{\delta}{\varepsilon\log\delta}\int_0^t \int_0^\infty G(s,z)^2  \nu(dz)ds + \int_0^t \int_u^\infty  z |(\overline{X}^n_s)^{\rho} - (X_{\eta(s)}^{n})^{\rho}| \nu(dz)ds \Big)\\
&\leq C\Big( \frac{\delta}{\varepsilon\log\delta}\int_0^t \int_0^\infty G(s,z)^2   \nu(dz)ds + \int_0^t  |\overline{X}^n_s - X_{\eta(s)}^{n}|^{\alpha_5\rho} ds \Big).
\end{align*}
where in the last inequality, we used \eqref{IJ}. By Lemma \ref{LocalLocal} and Lemma \ref{martingale:components}, the expected value of the above is bounded by 
\begin{align}
C\frac{\delta}{\varepsilon\log\delta}(n^{-\frac{1}{\alpha}}+ n^{-\frac{\alpha_5\rho}{2}}) \label{K332}
\end{align}
 where $\alpha_5 \in [0,2]\cap (0,\alpha/\rho)$.  Similarly, for $z< u$, we apply the first-order Taylor expansion and Young's inequality to obtain
\begin{align*}
	&\int_0^t \int_0^{1} G(s,z) \Big\{\phi_{\delta,\varepsilon}'(\overline{Y}^{ n}_{s-}+ \sigma_2^n[ X_{s-}^{\rho}  - (\overline{X}^n_s )^{\rho}]z ) -\phi_{\delta,\varepsilon}'(\overline{Y}^{ n}_{s-} + \sigma_2^n[X_{s-}^{\rho} - (X_{\eta(s)}^{n})^{\rho} ]z)\Big\} \nu(dz)ds\\
	&\leq C \int_0^t \int_0^u   \frac{\delta}{\varepsilon\log\delta} \big( G(s,z)^2 +z^2 |(\overline{X}^n_s)^{\rho}- (X_{\eta(s)}^{n})^{\rho} |^2 \big) \nu(dz)ds\\
	&\leq C \Big( \frac{\delta}{\varepsilon\log\delta} \int_0^t \int_0^\infty    G(s,z)^2 \nu(dz)ds  + \frac{\delta}{\varepsilon\log\delta} \int_0^t |(\overline{X}^n_s)- X_{\eta(s)}^{n} |^{\alpha_5\rho}ds  \Big).
\end{align*}
where in the last inequality, we again used \eqref{IJ}. By Lemma \ref{LocalLocal} and Lemma \ref{martingale:components}, the expected value of the above is again bounded by 
\begin{gather}
C\frac{\delta}{\varepsilon\log\delta}(n^{-\frac{1}{\alpha}}+ n^{-\frac{\alpha_5\rho}{2}}) \label{K333}
\end{gather}
 where $\alpha_5 \in [0,2]\cap (0,\alpha/\rho)$. By combining \eqref{K331}, \eqref{K332} and \eqref{K333}, we have for $\alpha_5 \in [0,2]\cap (0,\alpha/\rho)$, 
\begin{gather}
	\E[|K_t^{n,\delta,\varepsilon,3,3}|] \leq C \Big(\frac{\delta}{\varepsilon\log\delta}\,( n^{-\frac{1}{\alpha}} 
	+ n^{-(\rho\wedge \frac{\alpha_-}{2})}+n^{-\frac{\alpha_5\rho}{2}}) + n^{-\frac{\rho}{2}}\Big). \label{eq41}
\end{gather}

\vskip5pt

\noindent {\bf Optimising the convergence rate:}  Now we are ready to combine all the previous calculations to estimate the strong error $\E[|X_t - X_t^{n}|]$. We have from \eqref{eqI}, \eqref{EJ}, \eqref{eq35}, \eqref{eq36}, \eqref{eq37}, \eqref{eq40} and \eqref{eq41}
\begin{align*}
&\E[|X_t - X_t^{n}|] \leq  Cn^{-1} + \varepsilon
		+C_T \Big\{\int_0^t \E[|X_s -  X_{s}^{n}|] ds \  +  n^{-1/2} 
		+\frac{\varepsilon}{\log \delta}\\
&+ \frac{\delta}{\varepsilon \log \delta} ( n^{-1/\alpha}+n^{-\alpha_-/2}) + \varepsilon \log(\delta)^{\alpha -1} \big( I^{\alpha_1}_{\log(\delta)} + J^{\alpha_1}_{\log(\delta)}  \big) 
		+\frac{\delta}{\varepsilon \log \delta}\, n^{-\frac{\alpha_2\rho}{2}}\\
& +\Big(\Big(\frac{\delta}{\log\delta}\Big)^{\alpha_3-1} + 1\Big) n^{-\frac{\rho}{2}}
		+\frac{\delta}{\varepsilon\log\delta} n^{-\frac{1}{\alpha}}
		+\frac{\delta}{\varepsilon\log\delta}\, n^{-\frac{1}{\alpha}}+\frac{\varepsilon}{\log\delta}\big(n^{-\frac{1}{\alpha}}\varepsilon^{-2} \log\delta\big)^{\frac{2-\alpha_4}{3-\alpha_4}}\\
&+\frac{\delta}{\varepsilon\log\delta}\,(n^{-\frac{1}{\alpha}} + n^{-(\rho\wedge \frac{\alpha_-}{2})} +n^{-\frac{\alpha_5\rho}{2}}) + n^{-\frac{\rho}{2}}\Big\}. \nonumber 
\end{align*}
	Then, we can choose $\delta=2$, $\alpha_1=\alpha_3=\alpha$ and apply Gr{\"o}nwall's inequality to obtain
	\begin{align}
& \E[|X_t - X_t^{n}|]	\leq C_T \Big\{  n^{-1}+\varepsilon+ n^{-1/2} 
		+ \varepsilon + \varepsilon^{-1}  \big( n^{-\frac{\alpha_-}{2}} + n^{-\frac{1}{\alpha}}\big)
		+\varepsilon+\varepsilon^{-1} n^{-\frac{\alpha_2\rho}{2}}
		+ n^{-\frac{\rho}{2}} \nonumber \\
	&\qquad + \varepsilon^{-1}n^{-\frac{1}{\alpha}}+\varepsilon^{-1}n^{-\frac{1}{\alpha}}
		+ \varepsilon( \varepsilon^{-2} n^{-\frac{1}{\alpha}})^{\frac{2-\alpha_4}{3-\alpha_4}}
		+ \varepsilon^{-1}( n^{-\frac{1}{\alpha}} + n^{-(\rho\wedge\frac{\alpha_-}{2})} +n^{-\frac{\alpha_5\rho}{2}})+n^{-\frac{\rho}{2}} \Big\} \nonumber\\
	&\leq C_T \Big\{  n^{-1}+ n^{-\frac{1}{2}}   + n^{-\frac{\rho}{2}}  +\varepsilon + \varepsilon^{-1} (n^{-(\rho\wedge\frac{\alpha_-}{2})} + n^{-\frac{1}{\alpha}}
		+ n^{-\frac{\alpha_2\rho}{2}}+n^{-\frac{\alpha_5\rho}{2}}) + \varepsilon (\varepsilon^{-2}n^{-\frac{1}{\alpha}})^{\frac{2-\alpha_4}{3-\alpha_4}} \Big\}.\label{firstRate}
	\end{align}
We point out that, in the above, the rate $1/\alpha$ comes from the term $\overline{M}^n$ in the remainder $R^n$, the rate $\rho/2$ is related to the H\"older exponent of the jump coefficient, the rate $1/2$ is related to the linear drift term and the rate $\alpha_-/2$ essentially comes from maximizing the integrability that is available to us.

	To optimize the convergence rate in \eqref{firstRate}, we must select $\alpha_2$, $\alpha_4$ and $\alpha_5$ within the constraints
	$\alpha_2\in [\alpha,2]\cap (0,\alpha/\rho)$,
$\alpha_4\in [\alpha,2]$ and
	$\alpha_5\in [0,2]\cap (0,\alpha/\rho)$.
	Clearly we want to select $\alpha_2$ and $\alpha_5$ to be as large as possible and since both $\alpha_2$ and $\alpha_5$ has the same constraint for it's upper value, we can combine them into one term and write
\begin{align*}
\E[|X_t-X^n_t|]	&\leq C_T \Big\{  n^{-\frac{\rho}{2}} + \varepsilon + \varepsilon^{-1} (n^{-(\rho\wedge \frac{\alpha_2}{2})} + n^{-\frac{1}{\alpha}})
		+\varepsilon (\varepsilon^{-2}	n^{-\frac{1}{\alpha}})^{\frac{2-\alpha_4}{3-\alpha_4}} \Big\}.
\end{align*}

To this end, to find the optimal choice of $\varepsilon$, we further write the above into
\begin{equation*}
	\E[|X_t - X_t^{n}|]\leq C_T \big( n^{-\zeta_1} + n^{-r} + n^{-(\zeta_2-r)} + n^{-(\zeta_4+r\zeta_3)} \big),	
	\end{equation*}
where we set $\varepsilon=n^{-r}$ and
	\begin{align*}
	&\zeta_1 = \frac{\rho}{2},\quad
	\zeta_2 = \frac{\alpha_-}{2}\wedge \frac{1}{\alpha}\wedge \rho ,\quad
	\zeta_4 = \frac{2-\alpha_4}{\alpha(3-\alpha_4)}, \quad
	\zeta_3 = \frac{(\alpha_4-1)}{(3-\alpha_4)}.
	\end{align*}
	Consider now the term $n^{-(\zeta_4+r\zeta_3)}$.
	In order to maximize $(\zeta_4+r\zeta_3)$, we choose the optimal value of $\alpha_4$ which depends on the value of $r$. To do this, we see that 
\begin{gather*}
(\zeta_4+r\zeta_3) = r \frac{(\alpha_4-1)}{(3-\alpha_4)} + \frac{2-\alpha_4}{\alpha(3-\alpha_4)} = -r + \left(2r-\frac{1}{\alpha}\right) \frac{1}{3-\alpha_4} - \frac{1}{\alpha}
\end{gather*}	
	which is a monotonic function in $\alpha_4 \in [\alpha, 2]$. Then by maximizing the above with respect to $\alpha_4$, we obtain the following piecewise function
	\begin{equation*}
	f(r) =
	\begin{cases}
	\frac{(\alpha-1)}{(3-\alpha)}\, r + \frac{2-\alpha}{\alpha(3-\alpha)}, & r<\frac{1}{2\alpha},\\
	 r, &  r\geq  \frac{1}{2\alpha},
	\end{cases}	
	\end{equation*}	
which is greater than $r$. Therefore the optimal choice of $r$ is
\begin{gather*}
r^* = \underset{r\geq 0}{\arg\max}\,\, \Big(r \wedge (\zeta_2 -r ) \wedge  f(r)\Big) = \underset{r\geq 0}{\arg\max}\,\, \Big(r \wedge (\zeta_2 -r )\Big)
\end{gather*}
which is given by the intercept of $r$ and $\zeta_2 - r$, that is $r^* = \frac{1}{2}\zeta_2 = \frac{1}{2}\left(\frac{\alpha_-}{2}\wedge \rho\wedge \frac{1}{\alpha}\right) \leq \frac{\rho}{2}$. This shows that
\begin{gather*}
	\E[|X_t - X_t^{n}|] \leq C_Tn^{-\frac{1}{2}q(\alpha,\rho)}
\end{gather*}
where $q(\alpha,\rho) = \frac{\alpha_-}{2}\wedge\rho\wedge \frac{1}{\alpha}$ and $\alpha_- < \alpha$ can be chosen arbitrarily close to $\alpha$.
\end{proof}

\begin{cor}\label{t1c}
Suppose Assumption \ref{Assumption:3} holds and $\sigma_2 = 0$, then there exists a constant $C_T>0$ such that 
\begin{gather*}
	\sup_{t\leq T}\E[|X_t - X_t^{n}|] \leq C_Tn^{-\frac{1}{2}}.
\end{gather*}
\end{cor}
\begin{proof}
Again we write $\E[|X_t - \overline X_t^{n}|] + \E[|\overline R^n_t|]$ where $\E[|\overline R^n_t|] \leq C_Tn^{-1}$. The estimate of $I^{n,\delta, \epsilon}$ in \eqref{eqI} remains the same and we need only to modify the estimate in \eqref{EJ}. Since all processes involved are square integrable we can take $\alpha_- = 2$ in \eqref{EJ} and by Lemma \ref{martingale:components} we have $$\E[|\widehat M_t|^2] + \E[|\overline M_t|^2] = \E[\int^t_0 F^2(s) ds] = n^{-1}.$$ By combining these estimates, we obtain from Gr\"onwall's inequality
\begin{align*}
\E[|X_t - X_t^{n}|] \leq C_T \left\{ n^{-1} + \epsilon + n^{-\frac{1}{2}} + \frac{n^{-1}}{\epsilon} \right\}.
\end{align*}
Finally by setting $\epsilon = n^{-r}$ and noticing that the optimal $r$ is $r = 1/2$, we obtain the rate of $1/2$.
\end{proof}

\brem\label{remark1.5}
To conclude, we demonstrate here that, when $\rho = \frac{1}{\alpha}$, the rate obtained in Theorem \ref{Strong:Convergence} is an improvement over the rate obtained for the Euler-Maruyama scheme in \cite{FL}. To do this, let $q(\alpha,\rho) = q(\alpha) := \frac{\alpha_-}{2}\wedge \frac{1}{\alpha}$, we note that for $\alpha \in (1,2)$, 
$$\frac{\alpha_-}{4}\geq \frac{1}{2}q(\alpha) > \frac{(2-\alpha)}{(3-\alpha)}q(\alpha),$$
 and we show in the following that $n^{-\frac{(2-\alpha)}{(3-\alpha)}q(\alpha)}$ is faster than the order of convergence obtained in Theorem 2.11 of \cite{FL} which, in the case of the alpha-CEV, is given by $n^{-\rho/2}+ \varepsilon_n$ where
	\[
	\varepsilon_n =
	\left\{
	\begin{aligned}
	&n^{-\rho/2}+ n^{-(\rho-\frac{\rho}{2\gamma})},
	&
	&\text{ if } \alpha\in[1,\frac{2(1-\gamma)}{1-\rho}] \text{ and } \rho\leq\gamma,\\
	&n^{-(\gamma-\frac{1}{2})},
	&
	&\text{ if } \alpha\in[1,\frac{2(1-\gamma)}{1-\rho}] \text{ and } \rho>\gamma,\\
	&n^{-\rho(1-\frac{1}{2-\alpha(1-\rho)})},
	&
	&\text{ if } \alpha\in (\frac{2(1-\gamma)}{1-\rho},2],
	\end{aligned}
	\right.
	\]
	and $\rho = \frac{1}{\alpha}$. Note that we do not consider the case $\alpha = 1$ or $\alpha=2$ as they are not covered by Theorem \ref{Strong:Convergence}.

In the first case where $\alpha\in (1, 2(1-\gamma)/(1-\rho) ]$ and $\rho\leq\gamma$,
	we first write 
	\begin{equation}\label{FL1}
	\rho-\frac{\rho}{2\gamma}
	=
	\frac{\rho(2\gamma-1)}{(2\gamma-1)+1}.
	\end{equation}
	By assumption we have $2\gamma<\alpha$, from which
	one can deduce that $\rho\leq \gamma < \alpha/2$,
	and hence $\alpha^2>2$.
	On the other hand, we can similarly write
	\begin{equation}\label{CEV1}
	\frac{2-\alpha}{3-\alpha}\frac{1}{\alpha} =
	\frac{\rho(2-\alpha)}{(2-\alpha)+1}.
	\end{equation}
	Then using the condition $\alpha\leq 2(1-\gamma)/(1-\rho)$ we must have $2-\alpha\geq 2\gamma-1$.
	Together with the fact that 
	$x\mapsto \rho x/(x+1)$
	is an increasing function on $\R_+$,
we deduce that \eqref{CEV1} must be greater than \eqref{FL1}. In the second case where $\alpha \in (1, 2(1-\gamma)/(1-\rho) ]$ and $\rho > \gamma$, we deduce that $\alpha^2 \leq 2$ and one can similarly write
\begin{equation*}
	\gamma-\frac{1}{2}=
	\frac{\gamma(2\gamma-1)}{(2\gamma-1)+1} \quad \mathrm{and} \quad 
	\frac{2-\alpha}{3-\alpha}\frac{\alpha_-}{2}=
	\frac{\alpha_-}{2}\frac{(2-\alpha)}{(2-\alpha)+1}.
\end{equation*}
Then by using the fact that $2-\alpha\geq 2\gamma-1$, $\gamma<\rho$ and $\gamma<\frac{\alpha_-}{2}$ we deduce that
	\[\frac{\gamma(2\gamma-1)}{(2\gamma-1)+1}
	\leq \frac{\gamma(2-\alpha)}{(2-\alpha)+1}
	<\frac{\alpha_-}{2}\frac{(2-\alpha)}{(2-\alpha)+1}.\]
	In the third case where $\alpha \in (2(1-\gamma)/(1-\rho), 2)$,
	it is necessary that
	$2(1-\gamma)/(1-\rho)<2$, which implies that $\rho<\gamma$	and $\alpha^2>2$.
	Thus the two convergence rates in this case are equal, that is
	$\frac{2-\alpha}{\alpha(3-\alpha)}=\rho \big(1-\frac{1}{2-\alpha(1-\rho)}\big)$.

\erem

\section{Appendix}\label{CEV:Appendix}
\begin{proof}[{\bf Proof of Lemma \ref{CEV:positivity}}]
For $m \geq 0$ we set $\tau_m:= \inf \{t\geq 0 : X_t \notin(m^{-1},m)\}$. The aim is to show that for any fixed $T>0$, $\lim_{m\rightarrow \infty} \P(\tau_m \leq T) = 0$. To do this, for a fixed $0<\beta<1$, we consider the $C^2(\R_+,\R_+)$ function
$V(x) = x^\beta -1 -\beta\ln x$.
From the It\^o formula applied to $V(X_T^{\tau_m})$, we note that the drift term takes the form
\begin{align}\label{Strict:positivity:drift}
    & V'(x)(a  - k x)+ \frac{1}{2}V''(x)\sigma_1^2 x^{2\gamma} +\int_0^\infty [V(x+\sigma_2x^\rho z)-V(x) -\sigma_2x^\rho zV'(x)]\nu(dz),
\end{align}
where $V'(x) = \beta(x^{\beta -1}-x^{-1})$ and $V''(x)=\beta(\beta-1)x^{\beta-2}+\beta x^{-2}$.
	Thus we see that the terms in \eqref{Strict:positivity:drift} which are
	associated with the diffusion part of $X$,
	are given by
\begin{align*}
a\beta x^{-(1-\beta)}-k\beta x^\beta - a\beta x^{-1} +k\beta - \frac{\sigma_1^2\beta(1-\beta)}{2}x^{\beta-2(1-\gamma)} + \frac{\sigma_1^2\beta}{2} x^{-2(1-\gamma)}.
\end{align*}
From the fact that $0<\beta<1$ and $1/2< \gamma < 1$, we see that for $k >0$ the terms with the smallest and largest power are $-a \beta x^{-1}$ and $-k\beta x^\beta$ respectively, both of which have negative constant coefficients.

On the other hand, for the integral against the L{\'e}vy  measure we can apply Taylor's expansion to obtain, for $(x, y, z) \in \mathbb{R}_+^3$, the following useful equalities
\begin{align}
&V(x+\sigma_2x^\rho z)-V(x) -\sigma_2x^\rho zV'(x)=\sigma_2^2x^{2\rho} z^2\int_0^1 (1-\theta)V''(x+\theta\sigma_2x^\rho z)d\theta, \label{1stTaylor}\\
&V(x+\sigma_2x^\rho z)-V(x) -\sigma_2x^\rho zV'(x)=\sigma_2 x^\rho z\int_0^1 [V'(x+\theta\sigma_2x^\rho z)-V'(x)]d\theta. \label{2ndTaylor}
\end{align}
For $0< z < 1$, by using \eqref{1stTaylor} we obtain
\begin{align*}
&\int_0^1 [V(x+\sigma_2x^\rho z)-V(x) -\sigma_2x^\rho zV'(x)]\nu(dz)\\
&=\sigma_2^2 x^{2\rho}\int_0^1 z^2\int_0^1 (1-\theta)\, 
[\beta (\beta-1)(x+\theta\sigma_2x^\rho z)^{\beta-2}+\beta (x+\theta\sigma_2x^\rho z)^{-2}] \,
d\theta\, \nu(dz)\\
&\leq\sigma_2^2x^{2\rho}\int_0^1 z^2\int_0^1 (1-\theta)\, 
\beta\, x^{-2} \,
d\theta\, \nu(dz) = \frac{\beta\sigma_2^2}{2(2-\alpha)}\, x^{2\rho-2}.
\end{align*}
For $z\geq 1$, by using \eqref{2ndTaylor} we obtain
\begin{align*}
&\int_1^\infty [V(x+\sigma_2x^\rho z)-V(x) -\sigma_2x^\rho zV'(x)]\nu(dz)\\
&= \beta \sigma_2 x^\rho\int_1^\infty z\int_0^1  [(x+\theta\sigma_2x^\rho z)^{\beta-1}- (x+\theta\sigma_2x^\rho z)^{-1} - x^{\beta-1}+x^{-1}]\, d\theta\, \nu(dz)\\
&\leq \beta \sigma_2 x^\rho \int_1^\infty z\int_0^1  [x^{\beta-1}+x^{-1}]\, d\theta\, \nu(dz) = \frac{\beta\sigma_2}{\alpha-1}\,  (x^{\rho +\beta -1}+x^{\rho-1}).
\end{align*}
By combining the above computations, we observe that
\begin{align}
& V'(x)(a  - k x)+ \frac{1}{2}V''(x)\sigma_1^2 x^{2\gamma} + \int_0^\infty [V(x+\sigma_2x^\frac{1}{\alpha}z)-V(x) -\sigma_2x^\frac{1}{\alpha}zV'(x)]\nu(dz)\nonumber \\
& \leq \big(  a\beta x^{-(1-\beta)} -k\beta x^\beta -a\beta x^{-1} + k\beta - \frac{\sigma_1^2\beta(1-\beta)}{2}x^{\beta-2(1-\gamma)} + \frac{\sigma_1^2\beta}{2} x^{-2(1-\gamma)} \big) \nonumber \\
& \quad + \frac{\beta\sigma_2^2}{2(2-\alpha)}\, x^{-2(1-\rho)} + \frac{\beta\sigma_2}{\alpha-1}\, (x^{\beta - (1-\rho)}+x^{-(1-\rho)}). \label{Vdrift}
\end{align}
Using the fact that $1<\alpha<2$ and $\rho \in (1/2,1)$, we see that the terms with the smallest and largest power in \eqref{Vdrift} are still $-a\beta x^{-1}$ and $-k\beta x^\beta$ respectively, and both have negative constant coefficients. Therefore, the right hand side of \eqref{Vdrift} is a continuous function which is bounded from above. That is we can find a constant $C>0$ such that for all $0<x<\infty$
\begin{displaymath}
V'(x)(a  - k x)+ \frac{1}{2}V''(x)\sigma_1^2 x^{2\gamma} + \int_0^\infty [V(x+\sigma_2x^\rho z)-V(x) -\sigma_2x^\rho zV'(x)]\nu(dz) \leq C.
\end{displaymath}
Then by taking the expectation we obtain for every $m\geq 0$, $\E[V(X^{\tau_m}_T)] \leq V(x_0) + CT.$

Finally, since the function $V$ is
	strictly decreasing on $(0,1)$, strictly increasing on $(1,\infty)$, we can have for example that $V(m) \leq \E[V(X^{\tau_m}_T) | \tau_m\leq T, X_{\tau_m}\geq m]$, and conclude that
\begin{align*}
\P(\tau_m\leq T)\,[V(m^{-1})\wedge V(m)] & \leq \P(\tau_m\leq T, X_{\tau_m}\leq m^{-1}) \E[V(X^{\tau_m}_T) | \tau_m\leq T, X_{\tau_m}\leq m^{-1}]\\
&\qquad + \P(\tau_m\leq T, X_{\tau_m}\geq m) \E[V(X^{\tau_m}_T) | \tau_m\leq T, X_{\tau_m}\geq m]\\
&=\P(\tau_m\leq T) \E[V(X^{\tau_m}_T) | \tau_m\leq T] \\
&  \leq \E[V(X^{\tau_m}_T)].
\end{align*}
Hence, for any $m \geq 0$, we have the following estimates
\[\P(\tau_m\leq T)\,[V(m^{-1})\wedge V(m)]\leq \E[V(X^{\tau_m}_T)] \leq V(x_0) + CT.\]
Finally, the function $V(x)$ diverges to infinity as $x \rightarrow 0$ or $x \rightarrow \infty$. Therefore, by taking the limit as $m \rightarrow \infty$, one can conclude that for any $T > 0$, $\lim_{m\rightarrow\infty}\P(\tau_m\leq T)=0$.
\end{proof}

\begin{proof}[{\bf Proof of Lemma \ref{CEV:inverse:moments}}]
	In the following, by extending the argument in Lemma 4.1 of Bossy and Diop \cite{BossyDiop}, we apply the It\^o formula to estimate the inverse moments of $X$.	Let us consider a sequence of stopping times $(\tau_m)_{m\in\mathbb{N}^+}$ given by $\tau_m = \inf\{s\leq T: X_s\leq m^{-1}\}$.
	Then by applying the It{\^o} formula, see for example Theorem 4.4.7 in \cite{Applebaum}, to the stopped process $(X^{\tau_m})^{-p}$, we obtain
	\begin{equation*}
	(X_{t\wedge \tau_m})^{-p} = x_0^{-p} + M_{t\wedge \tau_m}^i +  I_{t\wedge \tau_m}^i+  J_{t\wedge \tau_m}^i +  K_{t\wedge \tau_m}^i,
	\end{equation*}
	where we set
\begin{align*}
M_{t}^i &:= - \sigma_1 p\int_0^t X_s^{\gamma-p-1}dW_s
+ \int_0^t\int_0^{\infty} \{ (X_{s-}+\sigma_2(X_{s-})^{\frac{1}{\alpha}}z)^{-p} - (X_{s-})^{-p} \} \widetilde{N}(ds,dz),\\
I_{t}^i &:= -p\int_0^t X_s^{-p-1} (a-kX_s)ds,\\
J_{t}^i &:= \frac{\sigma_1^2}{2}p(p+1)\int_0^t X_s^{2\gamma-p-2}ds,\\
	K_{t}^i &:= \int_0^t\int_0^{\infty} \{ (X_{s-}+\sigma_2(X_{s-})^{\rho}z)^{-p} - (X_{s-})^{-p} 
		+ \sigma_2(X_{s-})^{\rho}z p (X_{s-})^{-p-1}\} \nu(dz)ds.
	\end{align*}

We consider now the term $K_{t\wedge \tau_m}^i$.
	For $z\in (0,1)$ we deduce, from the first-order Taylor expansion of the function $\theta \mapsto (y+\theta xz)^{-p}$, the following inequality for every positive $x,y$ and $z$,
	\begin{align*}
&(y+ xz)^{-p} - y^{-p} + pxzy^{-p-1}\\
&=
p(p+1)(xz)^2 \int_0^1 (y+\theta xz)^{-p-2} (1-\theta) d\theta\leq
\frac{p(p+1)(xz)^2}{y^{p+2}}.
\end{align*}
For $z\geq 1$, since $x, y , z$ are positive we have $(y+ xz) \geq y$ and
\[
	(y+ xz)^{-p} - y^{-p} + pxzy^{-p-1}
	\leq
	pxzy^{-p-1}.
\]
	Hence we have the estimate
	\begin{align*}
		\E [K_{t\wedge \tau_m}^i] 
		& \leq
		\sigma_2^2 p(p+1)\E\big[\int_0^{t\wedge \tau_m}\int_0^{1} X_s^{2\rho-p-2}z^2 \nu(dz)ds\big] 
		+ p\sigma_2 \E\big[\int_0^{t\wedge \tau_m}\int_1^{\infty}
		+ 	X_s^{\rho-p-1}z \nu(dz)ds\big]\\
	    &=  \frac{\sigma_2^2 p(p+1)}{2-\alpha}  \E\big[\int_0^{t\wedge \tau_m} X_s^{2\rho-p-2}ds\big]
		+  \frac{ p\sigma_2}{\alpha-1} \E\big[\int_0^{t\wedge \tau_m} X_s^{\rho-p-1} ds\big]
.
	\end{align*}
	Similarly for $I_{t\wedge \tau_m}^i$ and $J_{t\wedge \tau_m}^i$ we have
\begin{align*}
\E[I_{t\wedge \tau_m}^i] &=
- p \E\big[\int_0^{t\wedge \tau_m} X_s^{-p-1} (a-kX_s)ds\big] \\
&= - ap \E\big[\int_0^{t\wedge \tau_m} X_s^{-p-1} ds \big] +pk \E\big[\int_0^{t\wedge \tau_m} X_s^{-p} ds\big], \\
\E[J_{t\wedge \tau_m}^i] &=
\frac{\sigma_1^2}{2}p(p+1) \E \big[ \int_0^{t\wedge \tau_m} X_s^{2\gamma-p-2}ds\big].
\end{align*}
Thus by combining the above estimates, we obtain
\begin{align}
\E[(X_{t\wedge \tau_m})^{-p}]
& =x_0^{-p}  +  \E[I_{t\wedge \tau_m}^i] +  \E[J_{t\wedge \tau_m}^i] +  \E[ K_{t\wedge \tau_m}^i] \nonumber\\
& \leq
x_0^{-p}
+ pk \E\big[\int_0^{t\wedge \tau_m} X_s^{-p} ds \big]+ \E \Big[ \int_0^{t\wedge \tau_m} \big(
-ap X_s^{-p-1} + \frac{\sigma_1^2}{2}p(p+1)  X_s^{2\gamma-p-2}\nonumber \\
&\qquad\qquad +\frac{\sigma_2^2 p(p+1)}{2-\alpha} X_s^{2\rho-p-2}
+ \frac{p\sigma_2}{\alpha-1}  X_s^{\rho-p-1} 
\big) ds \Big]. \label{inverseComb}
\end{align}

From the fact that $\gamma>\frac{1}{2}$ and $\rho \in (1/2,1)$, we observe, on the right hand side of \eqref{inverseComb}, that $(-p-1)$ is the smallest among the negative powers given by $(-p-1)$, $(2\gamma-p-2)$, $(2\rho -p-2)$ and $(\rho-p-1)$.
Also, the term $X_s^{-p-1}$ is the only term with a negative coefficient.	Thus,
for $f:(0,\infty)\mapsto \R$
defined by
\[
	f(x):=
		-ap x^{-p-1} + \frac{\sigma_1^2p(p+1)}{2}  x^{2\gamma-p-2}
		+ \frac{\sigma_2^2 p(p+1)}{2-\alpha} x^{2\rho-p-2}
		+ \frac{p\sigma_2}{\alpha-1}  x^{\rho-p-1},
\]
there exists a positive constant $C$ such that $f(x)\leq C$ for all $x>0$.
Therefore we obtain $\E[(X_{t\wedge \tau_m})^{-p}] \leq x_0^{-p}+ CT+ pk\int_0^{T} \E [(X_{s\wedge \tau_m})^{-p}] ds$,
and finally Gr\"onwall's inequality gives us 
	\begin{equation}\label{inverBound}
	\E[(X_{t\wedge \tau_m})^{-p}] \leq
	(x_0^{-p}+ CT) \exp(pkT).
	\end{equation}


\noindent 	Lastly,
	since the constant $C$ is independent of $m$,
	the upper bound \eqref{inverBound} for $\E[(X_{t\wedge \tau_m})^{-p}]$ is independent of the localising sequence $(\tau_m)_{m\in\mathbb{N}^+}$ and by strict positivity of the processes, given in Lemma \ref{CEV:positivity}, we can conclude with the monotone convergence theorem and letting $m\rightarrow \infty$.
\end{proof}

\vskip2pt
\begin{proof}[{\bf Proof of Lemma \ref{finiteE}}]
		By taking expectation of the scheme in \eqref{Scheme:expanded} and removing the terms with negative coefficients we obtain that for some $C>0$
		\begin{align*}
		\mathbb{E}[X_{t_{i+1}}^{n}]
		&=
		\mathbb{E}[X_{t_{i}}^{n}] + (a-k_n\mathbb{E}[X_{t_{i}}^{n}])\Delta t + \mathbb{E}[A^n_{t_i}]\Delta t\\
  &\leq C\Delta t+ (1+C\Delta t)\mathbb{E}[X_{t_{i}}^{n}] + \frac{1}{4\kappa_0^2}\mathbb{E}[D^-_{t_{i}}].
		\end{align*}
Next, for the term $D_{t_{i}}^- = \max(-D_{t_{i}}, 0)$, we first observe that 
\begin{align*}
		D_{t_{i}}^{-}
	&=
		-\left(\sigma_1^2(X^{n}_{t_i})^{2\gamma-1} (\Delta W_{t_i})^2  +
		4(1+k\Delta t)(		X^{n}_{t_i} 
		+ \big(a - \frac{\sigma_1^2}{2}(X^{n}_{t_i})^{2\gamma-1}\big)\Delta t
		+ \sigma_2  ( X^{n}_{t_i})^{\rho}\Delta Z_{t_i})\right)\I_{\{D_{t_{i}}<0\}}\\
&\leq
			C\big(\frac{\sigma_1^2}{2} (X_{t_i}^{n})^{2\gamma-1}\Delta t-\sigma_2 (X_{t_i}^{n})^{\rho}\Delta Z_{t_i}\big)\I_{\{D_{t_{i}}<0\}}.
\end{align*}

		Next we take the expectation of the above inequality and apply H{\"o}lder's inequality. To do this,
		we choose constants $p,q>1$ such that $1/p + 1/q = 1$,
		with the choice of $q$ small enough that $\rho q <1$ and $(2\gamma-1)q<1$.
		Then we have
		\begin{align*}
			\E[D_{t_{i}}^-] 
		&\leq
		C\sigma_1^2\Delta t\,
		\E[(X_{t_i}^{n})^{(2\gamma-1)q}]^\frac{1}{q}
		\P(D_{t_{i}}<0)^\frac{1}{p}
		+ C\sigma_2\E[(X_{t_i}^{n})^{\rho q}|\Delta Z_{t_i}|^q]^\frac{1}{q}
		\P(D_{t_{i}}<0)^\frac{1}{p}\\
		&\leq 
					C\sigma_1^2\Delta t\,
		\E[X_{t_i}^{n}]^{2\gamma-1}
		\P(D_{t_{i}}<0)^\frac{1}{p}
		+C\sigma_2\E[X_{t_i}^{n}]^\rho \E[|\Delta Z_{t_i}|^q]^\frac{1}{q}
		\P(D_{t_{i}}<0)^\frac{1}{p},
		\end{align*}
		where we have also used independent increments and Jensen's inequality since $\rho q<1$ and $(2\gamma-1)q<1$.
		Then by Lemma \ref{probability:D:negative} and the inequality $e^{-x} \leq m! x^{-m}$ for $m \in \mathbb{N}_+$ we have
		\begin{align}
			\E[D_{t_{i}}^-] 
		&\leq 
			\big(
			C
		\sigma_1^2 \E[X_{t_i}^{n}]^{2\gamma-1}\Delta t
		+C\sigma_2\E[X_{t_i}^{n}]^\rho (\Delta t)^\frac{1}{\alpha}
		\big)
		\P(D_{t_{i}}<0)^\frac{1}{p} \label{ED}\\
		&\leq
		\big( C
		\sigma_1^2 \E[X_{t_i}^{n}]^{2\gamma-1}\nonumber
		+C\sigma_2\E[X_{t_i}^{n}]^\rho \big)\Delta t\\
		&\leq C(1+ \E[X_{t_i}^{n}])\Delta t \nonumber
		\end{align}
for some $C>0$,
		 and the recursive equation $1+\mathbb{E}[X_{t_{i+1}}^{n}]
		 \leq
		 (1+C'' \Delta t)(1+\mathbb{E} [X_{t_i}^{n}])$
		 for some $C'' >0$.
		 This gives $1+\mathbb{E}[X_{t_{i}}^{n}]\leq (1+x)(1+C'' \Delta t)^n
		 \leq (1+x)e^{C''T}$.
\end{proof}

\begin{proof}[{\bf Proof of Lemma \ref{2.3}}]
		For the first moment, from \eqref{ED} we see that there exists constants $C>0$ and $p>1$ such that
		\[\mathbb{E}[D_{t_{i}}^-]\leq
		C\mathbb{P}(D_{t_{i}}<0)^{\frac{1}{p}}(1+\mathbb{E}[X_{t_i}^{n}])
		\Delta t^{\frac{1}{\alpha}}.
		 \]		
		For $\beta >1$,
		similar to the proof of Lemma \ref{finiteE}, we again observe that
		\begin{align*}
			D_{t_{i}}^- \leq
			C\big(\frac{\sigma_1^2}{2} (X_{t_i}^{n})^{2\gamma-1}\Delta t-\sigma_2 (X_{t_i}^{n})^{\rho}\Delta Z_{t_i}\big)\I_{\{D_{t_{i}}<0\}}.
		\end{align*}
		Then by choosing a constant $q>1$ small enough so that $(2\gamma-1)q<1$ and $\beta q<\alpha$ and $\rho q < 1$,
		we can apply H{\"o}lder's inequality with $1/q+1/p=1$ to obtain
\begin{align*}
\mathbb{E}[(D_{t_{i}}^-)^\beta]
	&\leq	C \E \big[\big((X_{t_i}^{n})^{(2\gamma-1)\beta}\Delta t^\beta + (X_{t_i}^{n})^{\rho\beta}(\Delta Z_{t_i})^\beta\big)\I_{\{D_{t_{i}}<0\}}\big]\\
	&\leq 	C \P(D_{t_{i}}<0)^{\frac{1}{p}} \Big( \E[(X_{t_i}^{n})^{(2\gamma-1)\beta q}]^\frac{1}{q}\Delta t ^\beta
		+\mathbb{E} [(X_{t_i}^{n})^{\rho\beta q}]^{\frac{1}{q}}
			\mathbb{E} [|\Delta Z_{t_i}|^{\beta q}]^{\frac{1}{q}}
		\Big)\\
	&\leq 	C\P(D_{t_{i}}<0)^{\frac{1}{p}}\Big( \E[(X_{t_i}^{n})^{(2\gamma-1)\beta q}]^\frac{1}{q}\Delta t ^\beta
		+\mathbb{E} [(X_{t_i}^{n})^{\rho\beta q}]^{\frac{1}{q}}
			\Delta t^\frac{\beta}{\alpha}\Big).
\end{align*}
		Given the choice of $q$ here we can apply growth condition of order $\beta$ to obtain
		\begin{align*}
		\mathbb{E}[|D_{t_{i}}^-|^\beta]
		&\leq
			C\P(D_{t_{i}}<0)^{\frac{1}{p}}
		(1+\E[(X_{t_i}^{n})^{\beta}])(\Delta t ^\beta
		+
			\Delta t^\frac{\beta}{\alpha}) = C\P(D_{t_{i}}<0)^{\frac{1}{p}}
		(1+\E[(X_{t_i}^{n})^{\beta}])	\Delta t^\frac{\beta}{\alpha}
		\end{align*}
which concludes the proof.
	\end{proof}

\begin{proof}[{\bf Proof of Lemma \ref{probability:D:negative}}]
In the following, to emphasis the dependence of $D_{t_{i}}$ on $X^n_{t_i}$ we write $D_{t_{i}} = D_{t_{i}}(X_{t_i})$. By independent increment we have
\begin{align*}
\P(D_{t_{i}}(X^{n}_{t_i}) < 0 \,|\,\F_{t_i}) = \P(D_{t_{i}}(x) < 0 ) \big|_{x= X^{n}_{t_i}}.
\end{align*}
To proceed, for $x \geq 0$,	we observe that
\begin{equation}\label{set:W:Z:jump}
\{ D_{t_{i}}(x) < 0\} = \Big\{\sigma_1^2x^{2\gamma-1} (\Delta W_{t_i})^2 + 4(1+k\Delta t)(x + (a - \frac{\sigma_1^2}{2}x^{2\gamma-1})\Delta t + \sigma_2x^\rho\Delta Z_{t_i}) < 0\Big\},
\end{equation}
and we need only to consider the case $x>0$ since the above set is empty when $x = 0$. Using the fact that the increments $\Delta W_{t_i}$ and $\Delta Z_{t_i}$ are independent and the scaling property of the Brownian motion and the stable process we can work with independent variables $Y\sim\chi^2_1$ and $Z\sim\mathrm{Stable}(\alpha,1,0, 1)$. We introduce the following set $A$ which is equal in probability to the set given in \eqref{set:W:Z:jump} 
\begin{align*}
		A &= \Big\{\Delta t Y + \frac{4}{\sigma_1^2}(1+k\Delta t)\big(x^{2-2\gamma}
		+ (\frac{a}{x^{2\gamma-1}}  - \frac{\sigma_1^2}{2})\Delta t + \sigma_2x^{1+\rho-2\gamma}\Delta t^\frac{1}{\alpha} Z\big) < 0 \Big\}\\
		& =  \Big\{Y < 2(1+k\Delta t)\Big[1 - \frac{2}{\sigma^2_1}\big(\frac{x^{2-2\gamma}}{\Delta t }
		+ \frac{a}{x^{2\gamma-1}}+ \frac{\sigma_2x^{1+\rho-2\gamma}}{\Delta t}\Delta t^\frac{1}{\alpha} Z \big)\Big]\Big\}.
\end{align*}
Since the chi-square variable $Y$ is non-negative, the set above is non-empty only if the right hand side in the above is positive. In other words,
\begin{align}
A \subseteq \Big\{\sigma_2Z <
	\frac{\sigma_1^2}{2}\frac{\Delta t^{1-\frac{1}{\alpha}}}{x^{1+\rho-2\gamma}} \Big(1 - \frac{2}{\sigma^2_1} \Big( \frac{x^{2-2\gamma}}{\Delta t }+ \frac{a}{x^{2\gamma-1}}  \Big)\Big) \Big\}. \label{setA}
\end{align}
To proceed, we consider $f:\R^+\mapsto\R$ given by
\[
f(x):=\frac{\sigma_1^2}{2}\frac{\Delta t^{1-\frac{1}{\alpha}}}{x^{1+\rho-2\gamma}} \left(1 - \frac{2}{\sigma^2_1} \left( \frac{x^{2-2\gamma}}{\Delta t }+ \frac{a}{x^{2\gamma-1}}  \right)\right).
\]
The goal now is to find an upper bound $K$, which depend on $\Delta t$, for the function $f$ so that we have $\{ Z < f(x) \} \subseteq \{ Z < K\}$. To do that we seek to first maximise $f$ on the set $\{x \geq \Delta t\}$ and then on $\{x < \Delta t\}$.

Firstly, on the set $\{x \geq \Delta t\}$, we rewrite $f$ into
\begin{align*}
f(x) 
& = \frac{\sigma_1^2}{2}\Delta t^{-\frac{1}{\alpha}} x^{1-\rho} F(x)
\end{align*}
where $F(x) = \left(x^{2\gamma-2}\Delta t- \frac{2}{\sigma^2_1} \left( 1 + \frac{a}{x}\Delta t  \right)\right)$. The function $x\mapsto F(x)$ achieve its unique maximum at the point $x^* = \big[\frac{2a }{\sigma^2_1(2-2\gamma)}\big]^\frac{1}{2\gamma-1}$ and by substituting it back into $F$ we see that
	\[
	F(x^*)= \Delta t\,\frac{2a}{\sigma_1^2}\,
	\frac{2\gamma-1}{2-2\gamma}
	\Big[\frac{2a }{\sigma^2_1(2-2\gamma)}\Big]^{-\frac{1}{2\gamma-1}}\, -\frac{2}{\sigma_1^2}.
	\]
	The above shows that it is possible to pick $\Delta t$ is small enough such that $F(x^*) < -c < 0$ for some $c>0$. Thus we have the upper bound
\begin{align}
f(x)\leq -\frac{\sigma_1^2 }{2} \Delta t^{-\frac{1}{\alpha}} x^{1-\rho} \, c
\leq -c\, \frac{\sigma_1^2}{2}\Delta t^{1-\frac{1}{\alpha}-\rho}. \label{boundZ}
\end{align}
where in the last inequality we have used the fact that $x^{1-\rho}\geq \Delta t^{1-\rho}$. We note that $1-\frac{1}{\alpha} - \rho < 0$ since it was assumed that $\alpha \in (1,2)$ and $\rho \in (1-\frac{1}{\alpha},1)$.

Secondly, on the set $\{x < \Delta t\}$, we rewrite $f$ into
\begin{align*}
f(x)  = \frac{\sigma_1^2}{2}\Delta t^{1-\frac{1}{\alpha}} x^{-\rho}H(x) 
\end{align*}
where
$H(x) := (\Delta t)^{-1}xF(x) = \left(  x^{2\gamma-1}- \frac{2}{\sigma^2_1} \frac{x}{\Delta t }  - \frac{2a}{\sigma^2_1} \right)$.
The function $x\mapsto H(x)$ achieves its unique maximum $x^*$ given by
$x^* = \big[\Delta t \frac{\sigma^2_1(2\gamma-1)}{2}\big]^\frac{1}{2-2\gamma}$ which when substituted back into $H$ gives
\begin{align*}
H(x^*)
	=\Delta t^\frac{2\gamma-1}{2-2\gamma}\frac{4(1-\gamma)}{\sigma_1^2(2\gamma-1)}\Big[\frac{\sigma_1^2(2\gamma-1)}{2}\Big]^\frac{1}{2-2\gamma} -\frac{2a}{\sigma_1^2}.
\end{align*}
Again $\exists c>0$ such that $H(x^*)<-c<0$ when $\Delta t$ is small enough and by using $x^{-\rho} > \Delta t^{-\rho}$, we have 
\begin{align}
f(x) \leq \frac{\sigma_1^2}{2}H(x^*) \Delta t^{1-\frac{1}{\alpha}} x^{-\rho} \leq  -c\frac{\sigma_1^2}{2} \Delta t^{1-\frac{1}{\alpha}-\rho}.\label{boundZZ}
\end{align}
By combining \eqref{boundZ} and \eqref{boundZZ}, we see that there exists some constant $C>0$ such that
\begin{align*}
f(x) \leq  - C\Delta t^{1-\frac{1}{\alpha}-\rho}.
\end{align*}
From this we deduce that 
\begin{align*}
 \P(D_{t_{i}}(x) < 0 )\,\big|_{x = X^{n}_{t_i}}\   \leq\  \P(\sigma_2 Z < f(x))\big|_{x = X^{n}_{t_i}}\   \leq\  \P\big(Z <- C\Delta t^{1-\frac{1}{\alpha}-\rho}\big).
\end{align*}

Finally, by using the inequality $\I_{\{x < K\}} \leq e^{-(x-K)}$ we deduce that
	\[\P(D_{t_{i}} < 0 \,|\,\F_{t_i}) 
	\leq 
	\mathbb{E}[e^{-Z}]\exp(-C\Delta t^{1-\frac{1}{\alpha}-\rho}),\]
	where the quantity $\mathbb{E}[e^{-Z}]$ is finite since it is the Laplace transform of a stable random variable and is given in Lemma \ref{prop:Laplace:alpha:stable}.

In the case $\sigma_2= 0$, let us consider again \eqref{setA} and observe that 
\begin{align*}
\Big\{\sigma_2Z <
	\frac{\sigma_1^2}{2}\frac{\Delta t^{1-\frac{1}{\alpha}}}{x^{1+\rho-2\gamma}} \Big(1 - \frac{2}{\sigma^2_1} \Big( \frac{x^{2-2\gamma}}{\Delta t }+ \frac{a}{x^{2\gamma-1}}  \Big)\Big) \Big\} = \Big\{0 <
	\frac{\sigma_1^2}{2}\frac{\Delta t^{1-\frac{1}{\alpha}}}{x^{1+\rho-2\gamma}} \Big(1 - \frac{2}{\sigma^2_1} g(x)\Big) \Big\}
\end{align*}
where the function $g:\R^+\mapsto\R$ given by $g(x):=x^{2-2\gamma}/\Delta t + a/ x^{2\gamma-1}$. The unique minimiser $x^*$ of $g$ is given by $x^*= \frac{a(2\gamma-1)}{(2-2\gamma)} \Delta t =: z\Delta t$ and the minimum of $g$ is given by $g(x^*)   = (  z^{2-2\gamma} + az^{1-2\gamma}) (\Delta t)^{1-2\gamma} =: C(\Delta t)^{1-2\gamma}$.
From this, we deduce that
\begin{align*}
\Big\{0 <
	\frac{\sigma_1^2\Delta t^{1-\frac{1}{\alpha}}}{2x^{1+\rho-2\gamma}} \big(1 - \frac{2}{\sigma^2_1} g(x)\big) \Big\}\subseteq \Big\{0 <
	\frac{\sigma_1^2\Delta t^{1-\frac{1}{\alpha}}}{2x^{1+\rho-2\gamma}} \big(1 - \frac{2}{\sigma^2_1} C(\Delta t)^{1-2\gamma}\big) \Big\}.
\end{align*}
Therefore by having $\Delta t \leq (2C/\sigma_1^2)^\frac{1}{2\gamma-1}$ we see that the set in the right-hand-side above is the empty set and $\P(D_{t_i}(X_{t_i}^{n})<0\,|\,\F_{t_i}) = 0$.
\end{proof}

\begin{proof}[{\bf Proof of Lemma \ref{scheme:beta:moments}}]
To proceed, we decompose the L\'evy process $Z$ into small jumps and large jumps, by size one,
and express the scheme as
\begin{align*}
			X_{t_{i+1}}^n
		&=
			x_0
			+ \int_0^{t_{i+1}} (a-k_n X_{\eta(s)}^{n} - \frac{\sigma_2}{\alpha-1}  ( X_{\eta(s)}^{n})^{\rho} )ds
			+ \int^{t_{i+1}}_0 \sigma_1 (X_{\eta(s)}^{n})^\gamma dW_{s} \nonumber\\
	&\hspace{0.56cm}+ 
	\widetilde{M}^n_{t_{i+1}}
	+ \widehat{M}^n_{t_{i+1}}
	+ \overline{M}^n_{t_{i+1}}
			+ \int^{t_{i+1}}_0 A_{\eta(s)}^n ds 
			\nonumber\\
		&\hspace{0.56cm}
			+ \int_0^{t_{i+1}}\int_0^1 \sigma_2( X_{\eta(s)}^{n})^{\rho} z \widetilde{N}(dz,ds)
			+ \int_0^{t_{i+1}}\int_1^\infty \sigma_2( X_{\eta(s)}^{n})^{\rho} z {N}(dz,ds).
	\end{align*}
For notational simplicity, in the following we set
$N^n
	:= \widetilde{M}^n
	+ \widehat{M}^n
	+ \overline{M}^n$
	and $b(x) := a-k_n x - \frac{\sigma_2}{\alpha-1}  x^{\rho}$.
	Also we set $V_{t} := \int_0^{t}\int_1^\infty \sigma_2( X_{\eta(s)}^{n})^{\rho} z {N}(dz,ds) $,
	and $I_j:=\{0, 1,...,j\}$ for $j = 0,1,\dots, n$. 

	The goal is to localise the discretisation scheme and apply Gr\"onwall's inequality.
	Before proceeding further,
	we note that
	for any continuous-time process $Y$
	observed on the time grid $\pi$,
	when $Y$ is stopped at a discrete random time $\tau$ which takes value on the time grid $\pi$, we have for point $t_i$ that
	\[
	Y^\tau_{t_i} - Y_0 = \sum_{j=0}^{i-1} \I_{\{t_{j}<\tau\}}\Delta Y_{t_j} 
	= \sum_{j=0}^{n-1} \I_{\{t_{j}<t_i\wedge \tau\}}\Delta Y_{t_j} 
	= \int^T_0 \I_{\{s \leq \tau\}}\I_{\{s \leq t_i\}}  dY_s,
	\]
where by convention $\sum_{j=0}^{-1}=  0$.
In addition, we observe that $Y_{t_i} = Y_{t_i}^{\tau}$ on the set $\{t_{i} \leq \tau\}$ and $\{t_{i} <  \tau\} = \{t_{i} \leq   \eta(\tau)\}$ which implies that $Y_{t_i} = Y^{\eta(\tau)}_{t_i}$ on the set $\{t_{i} <  \tau\}$. 
	
For fixed $N \in \mathbb{N}$ define $\tau_N := \min\{t_i \leq T: X^n_{t_i} \geq N\}$
and we consider the localising sequence $(\tau_N)_{N\in\mathbb{N}}$.
Then from Jensen's inequality, we have for $\beta \in (1,\alpha)$
\begin{align}
& \max_{i\in I_n} ( X^{n}_{t_i\wedge \eta(\tau_N)	} )^\beta 
\leq		\max_{i\in I_n} ( X^{n}_{t_i\wedge \tau_N} )^\beta \nonumber \\
&\leq C_T \Big[ x_0^\beta
+ \max_{i\in I_n}  | \sum_{j=0}^{i-1} \I_{\{t_j< \tau_N\}} b(X_{t_j}^n)\Delta t |^{\beta}
+ \max_{i\in I_n} |\sum_{j=0}^{i-1} \I_{\{t_{j}<\tau_N\}} A_{t_j}^n \Delta {t}|^{\beta}
\label{ee1}\\
&\ + \max_{i\in I_n}|\sum_{j=0}^{i-1} \I_{\{t_{j}<\tau_N\}} \sigma_1 (X_{t_j}^{n})^\gamma \Delta W_{t_j}|^\beta 
+  \max_{i\in I_n} |\sum_{j=0}^{i-1} \int^1_0 \I_{\{t_{j}<\tau_N\}}  \sigma_2(X_{t_j}^n)^{\rho}z  \widetilde{N}(dz, \Delta t) |^\beta
\label{ee2}\\
&\ + \max_{i\in I_n} |V_{t_i\wedge \tau_N}|^\beta
+ \max_{i\in I_n} |N_{t_i\wedge \tau_N}^{n}|^\beta \Big].\label{ee3}
\end{align}

As the rest of proof is long and somewhat repetitive, we explain here the main idea by first focusing on the term with $b(x)$ in \eqref{ee1}. To estimate this term, we note that $\I_{\{t_j <\tau_N\}}b(X_{t_j}^{n}) = \I_{\{t_j <\tau_N\}} b(X_{t_j\wedge \eta(\tau_N)}^{n})$ and by using the fact that $|b(x)| \leq 1+ |x|$ and the Jensen equality, we obtain
\begin{gather}
\max_{i\in I_n}  | \sum_{j=0}^{i-1} \I_{\{t_j< \tau_N\}} b(X_{t_j}^n)\Delta t |^{\beta} \leq C\sum_{j = 0}^{n-1} [ 1+ \max_{i\in I_j} (X_{t_i \wedge \eta(\tau_N)}^{n})^\beta ] \Delta t. \label{E0}
\end{gather}
From the above, one is ready to apply the discrete Gr\"onwall's inequality as we have appropriately stopped the summand at $\eta(\tau_N)$ to make sure that it is bounded.

\noindent {\bf Step 1:} Let us consider the terms in \eqref{ee1}, that is
	\[
	\max_{i\in I_n}  | \sum_{j=0}^{i-1} \I_{\{t_j< \tau_N\}} b(X_{t_j}^n)\Delta t |^{\beta}
	+ \max_{i\in I_n} |\sum_{j=0}^{i-1} \I_{\{t_{j}<\tau_N\}} A_{t_j}^n \Delta{t}|^{\beta}.
		\]
	For the process $A^n$, we have by Jensen's inequality that
	\[
		\E\big[ \max_{i\in I_n} |\sum_{j=0}^{i-1} \I_{\{t_{j}<\tau_N\}} A_{t_j}^n\Delta{t}|^{\beta} \big]
	\leq
		n^{\beta-1} \sum_{j=0}^{n-1} \E[|A_{t_j}^n\Delta{t}|^{\beta}\I_{\{t_{j}<\tau_N\}}].
	\]
One can then apply Jensen's inequality, Lemma \ref{2.3}, Lemma \ref{probability:D:negative} and the fact that, for all $x\geq 0$, we have  $e^{-x} \leq n!x^{-n}$ for $n \in \mathbb{N}_+$ to obtain
	\begin{align}
		\E[|A_{t_j}^n\Delta{t}|^{\beta}\I_{\{t_{j}<\tau_N\}}]
	&\leq
		C ( \Delta t^{2\beta} + \E [  \E [ D_{t_{j}}^- \I_{\{t_{j}<\tau_N\}}| \mathcal{F}_{t_j} ]^{\beta} ] )\nonumber \\
	&  \leq
		C\Big( \Delta t^{2\beta} +  	  ( 1+ \E[(X_{t_j\wedge \eta(\tau_N)}^{n})^{\beta}])(\Delta t^\frac{\beta}{\alpha})
		\exp (- c (\Delta t)^{-(\frac{1}{\alpha}+\rho -1)} )\Big) \label{Aestimate}\\
		&\leq
		C \Big( \Delta t^{2\beta} + \E[(X_{t_j\wedge \eta(\tau_N)}^{n})^{\beta}]\Delta t ^\beta \Big).\nonumber 
	\end{align}
This gives the estimate
\begin{align}
\E\big[ \max_{i\in I_n} |\sum_{j=0}^{i-1} \I_{\{t_{j}<\tau_N\}} A_{t_j}^n\Delta{t}|^{\beta} \big] 
&\leq  C_T \Big( n^{-\beta}
+ \sum_{j=0}^{n-1}\E[(X_{t_j\wedge \eta(\tau_N)}^{n})^{\beta}]\Delta t \Big).
\label{E1}
\end{align}

\noindent
{\bf Step 2: } Given any discrete martingale $M$ and $\beta \in(1,\alpha)$, Doob's maximal inequality gives
	\begin{equation*}
	\mathbb{E}\big[\max_{i\in I_n} |M^{\tau}_{t_i}|^\beta\big]
	\leq  C_\beta\mathbb{E}[|M_{\tau}|^{\beta}].
	\end{equation*}
	In addition, if the martingale $M$ is square integrable
	then by the discrete time Burkholder-Davis-Gundy inequality
	we obtain
	\begin{equation}\label{BDG12}
	\mathbb{E}\big[\max_{i\in I_n} |M^{\tau}_{t_i}|^\beta\big] 
	\leq
	C\mathbb{E} \big[[M]_{\tau}^{\beta / 2} \big]
	= C\mathbb{E}\big[\big|\sum^{n-1}_{j=0} \I_{\{t_j < \tau\}} (\Delta M_{t_{j}})^2 \big|^\frac{\beta}{2}\big].
	\end{equation}

To this end, we recall
	that
	$
	N^n
	= \widetilde{M}^n
	+ \widehat{M}^n
	+ \overline{M}^n
	$,
	where 
$\overline M^n, \widehat M^n$ are $L^2$-martingales
	and $\widetilde M^n$ is a martingale with finite moments up to but excluding $\alpha$.
	We recall that their increments are given by
	\begin{align*}
	& \Delta \widetilde{M}^n_{t_i} := \frac{\Delta M^D_{t_i}}{4(1+k\Delta t)^2} 
	\qquad \Delta \widehat{M}^n_{t_i} :=  \frac{\sigma_1^2 (X^{n}_{t_i})^{2\gamma-1}}{2(1+k\Delta t)^2}((\Delta W_{t_i})^2-\Delta t ), \\
	 &\Delta \overline{M}^n_{t_i} := \Delta M_{t_i}^n-\sigma_1 (X^{n}_{t_i})^\gamma  \Delta W_{t_i}.
	\end{align*}
We first consider $\overline{M}^n$ and by Assumption \ref{Assumption:2}, we have $\kappa_n :=1+k\Delta t \geq \kappa_0$. From the fact that $x\mapsto \sqrt{x^+}$ is H\"older continuous with H\"older exponent $1/2$ and direct computations, we have 
	\begin{align}
(\Delta \overline{M}^n_{t_i})^2
& =\frac{\sigma_1^2 (X_{t_i}^{n})^{2\gamma-1}(\Delta W_{t_i})^2}{4\kappa_n^4}\Big|\sqrt{D_{t_{i}}^+}- \sqrt{4\kappa_n^4(X_{t_i}^{n})} \Big|^2 \nonumber\\
	&\leq \frac{\sigma_1^2 (X_{t_i}^{n})^{2\gamma-1} (\Delta W_{t_i})^2}{4\kappa_n^4}\  \Big|\sigma_1^2 (X_{t_i}^{n})^{2\gamma-1} (\Delta W_{t_i})^2 \label{Mbarsq}\\
	&\qquad +4\kappa_n
\big[ (1-\kappa_n^3)X_{t_i}^{n} + D_{t_{i}}^- + \big( a - \frac{\sigma_1^2}{2}(X^{n}_{t_i})^{2\gamma-1}\big)\Delta t+ \sigma_2  ( X^{n}_{t_i})^{\rho}\Delta Z_{t_i}   \big]
	\Big| \nonumber\\
&\leq C\Big[(X_{t_i}^{n})^{4\gamma-2} (\Delta W_{t_i})^4 + n^{-1}(X_{t_i}^{n})^{2\gamma} (\Delta W_{t_i})^2
	+ (X_{t_i}^{n})^{2\gamma-1} (\Delta W_{t_i})^2 D_{t_{i}}^- \nonumber\\
&\qquad +n^{-1}(X_{t_i}^{n})^{2\gamma-1} (\Delta W_{t_i})^2
	+n^{-1} (X_{t_i}^{n})^{4\gamma-2} (\Delta W_{t_i})^2
	+(X_{t_i}^{n})^{2\gamma-1+\rho} (\Delta W_{t_i})^2 |\Delta Z_{t_i}|\Big]. \nonumber
	\end{align}
By making use of the indicator function $\I_{\{t_i < \tau_N\}}$ in \eqref{BDG12}, all terms involving $X^n_{t_i}$ in \eqref{Mbarsq} can be stopped at $\eta(\tau_N)$.
We then proceed similarly to Lemma 3.2 of Gy{\"o}ngy and R{\'a}sonyi \cite{GyongyRasonyi} and take out a maximum from the sum to control powers. That gives us
\begin{align}
&\big| \sum^{n-1}_{j=0} \I_{\{t_j < \tau_N\}} (\Delta \overline{M}^n_{t_{j}})^2 \Big|^\frac{\beta}{2}\leq C\big|\frac{1}{ CK (1-\frac{\beta}{2} )}\max_{i\in I_n} (X_{t_i\wedge \eta(\tau_N)}^{n})^{\frac{\theta\beta}{2-\beta}} \Big|^{(1-\frac{\beta}{2})}\times \nonumber\\
&\quad\Big| \big[CK\big(1-\frac{\beta}{2}\big)\big]^{\frac{2}{\beta}(1-\frac{\beta}{2})}\sum^{n-1}_{j=0}  \Big[
(X_{t_j\wedge \eta(\tau_N)}^{n})^{4\gamma-2-\theta} (\Delta W_{t_j})^4 + n^{-1}(X_{t_j\wedge \eta(\tau_N)}^{n})^{2\gamma-\theta} (\Delta W_{t_j})^2\nonumber\\
&\qquad+ (X_{t_j\wedge \eta(\tau_N)}^{n})^{2\gamma-1-\theta} (\Delta W_{t_j})^2 D_{t_{j}}^-
+n^{-1}(X_{t_j\wedge \eta(\tau_N)}^{n})^{2\gamma-1-\theta} (\Delta W_{t_j})^2 \label{GR}\\
&\qquad+n^{-1} (X_{t_j\wedge \eta(\tau_N)}^{n})^{4\gamma-2-\theta} (\Delta W_{t_j})^2
+(X_{t_j\wedge \eta(\tau_N)}^{n})^{2\gamma-1+\rho-\theta} (\Delta W_{t_j})^2 |\Delta Z_{t_j}|\Big]\Big|^\frac{\beta}{2}, \nonumber
\end{align}
where $K>0$ is a constant to be later chosen. We choose $\theta \in ((2\gamma-\beta)^+, (2\gamma-1)\wedge (2-\beta))$ so that all powers of the scheme $X^n$ are positive and smaller than $\beta$.
For example, the fact that $\theta<2-\beta$ ensures that $\theta\beta/(2-\beta)<\beta$ in the first term of \eqref{GR}. On the other hand, in the sum, the largest power is $2\gamma$ and the condition $\theta>(2\gamma-\beta)^+$ guarantees that $2\gamma-\theta<\beta$, while for the smallest power $2\gamma-1$, we will have $2\gamma-1-\theta>0$. Then by applying Young's inequality with $p=2/(2-\beta)$ and $q=2/\beta$ to the right hand side of \eqref{GR}, we obtain 
\begin{align}
&\Big| \sum^{n-1}_{j=0} \I_{\{t_j < \tau_N\}} (\Delta \overline{M}^n_{t_{j}})^2 \Big|^\frac{\beta}{2} \leq
\frac{1}{ K }\max_{i\in I_n} (X_{t_i\wedge \eta(\tau_N)}^{n})^{\frac{\theta\beta}{2-\beta}}+C\, \frac{\beta}{2} \big[CK\big(1-\frac{\beta}{2}\big)\big]^{\frac{2}{\beta}(1-\frac{\beta}{2})}\times \nonumber \\
&\sum^{n-1}_{j=0}
    \Big[(X_{t_j\wedge \eta(\tau_N)}^{n})^{4\gamma-2-\theta} (\Delta W_{t_j})^4
    + n^{-1}(X_{t_j\wedge \eta(\tau_N)}^{n})^{2\gamma-\theta} (\Delta W_{t_j})^2 \nonumber\\
& \quad
		+ (X_{t_j\wedge \eta(\tau_N)}^{n})^{2\gamma-1-\theta} (\Delta W_{t_j})^2 D_{t_{j}}^-
		+n^{-1}(X_{t_j\wedge \eta(\tau_N)}^{n})^{2\gamma-1-\theta} (\Delta W_{t_j})^2\nonumber\\
& \quad +n^{-1} (X_{t_j\wedge \eta(\tau_N)}^{n})^{4\gamma-2-\theta} (\Delta W_{t_j})^2+(X_{t_j\wedge \eta(\tau_N)}^{n})^{2\gamma-1+\rho-\theta} (\Delta W_{t_j})^2 |\Delta Z_{t_j}|\Big]. \label{YY}
\end{align}

The next step is to take the expectation in equation \eqref{YY}. To do that, we first evaluate the term involving $D^-_{t_{j}}$. That is we consider the term
\begin{align*}
\E [(X_{t_j\wedge \eta(\tau_N)}^{n})^{2\gamma-1-\theta} (\Delta W_{t_j})^2 (-D_{t_{j}})\I_{\{D_{t_{j}} <  0\}}].
\end{align*}
We observe that 
\begin{align*}
	(X_{t_j\wedge \eta(\tau_N)}^{n})^{2\gamma-1-\theta}(\Delta W_{t_j})^2(-D_{t_{j}})
	&\leq 		
		\frac{\sigma_1^2}{2}(X^{n}_{t_j})^{4\gamma-2 - \theta}\Delta t (\Delta W_{t_j})^2
		+ \sigma_2  ( X^{n}_{t_j})^{2\gamma- 1 + \rho - \theta}\Delta Z_{t_i}(\Delta W_{t_j})^2
\end{align*}
By using H\"older's inequality with $\frac{1}{p} + \frac{1}{q} = 1$ and $q < \alpha$ together with Jensen's inequality we obtain
\begin{align*}
& 		\E\Big[\Big(\frac{\sigma_1^2}{2}(X^{n}_{t_j})^{4\gamma-2 - \theta}\Delta t (\Delta W_{t_j})^2
		+ \sigma_2  ( X^{n}_{t_j})^{2\gamma- 1 + \rho - \theta}\Delta Z_{t_j}(\Delta W_{t_j})^2\Big)^q \Big]^\frac{1}{q} \E[\I_{\{D_{t_{j}} <  0\}}]^\frac{1}{p} \\
	& \leq 	C \E\Big[(X^{n}_{t_j})^{(4\gamma-2 - \theta)q} (\Delta t)^q(\Delta W_{t_j})^{2q}
		+ ( X^{n}_{t_j})^{(2\gamma- 1 + \rho - \theta)q}(\Delta Z_{t_j})^q(\Delta W_{t_j})^{2q}\Big]^\frac{1}{q} \E[\I_{\{D_{t_{j}} <  0\}}]^\frac{1}{p} 
\end{align*}
We select $q\geq 1$ such that $(2\gamma- 1 + \rho - \theta)q < \beta$ and $(4\gamma- 2 - \theta)q < \beta$ which is possible since $\theta$ was previously chosen so that $(2\gamma- 1 + \rho - \theta) < \beta$ and $(4\gamma- 2 - \theta)< \beta$. From the inequality $|x|^a \leq 1+ |x|^b$ for $0<a \leq b$ the above can be further bounded by
\begin{align*}
& C \E\left[1+ (X^{n}_{t_j})^{(4\gamma-2 - \theta)q}
		+ (X^{n}_{t_j})^{(2\gamma- 1 + \rho - \theta)q} \right] \Delta t ^{1+\frac{1}{\alpha}} \E[\I_{\{D_{t_{j}} <  0\}}]^\frac{1}{p} \\
		& \leq C \big(1
		+ \E[|X^{n}_{t_j}|^{\beta}] \big) \Delta t ^{1+\frac{1}{\alpha}} \E[\I_{\{D_{t_{j}} <  0\}}]^\frac{1}{p}.
\end{align*}
By taking the expectation of \eqref{YY}, and
	using independence among the terms $X_{t_j}$, $\Delta W_{t_j}$ and $\Delta Z_{t_j}$, we obtain
\begin{align*}
&\E\big[\big| \sum^{n-1}_{j=0} \I_{\{t_j < \tau_N\}} (\Delta \overline{M}^n_{t_{j}})^2 \big|^\frac{\beta}{2}\big] \leq \frac{1}{ K } \E \big[\max_{i\in I_n} (X_{t_i\wedge \eta(\tau_N)}^{n})^{\frac{\theta\beta}{2-\beta}} \big]\\
&\quad+C \sum^{n-1}_{j=0} 
\Big[ \E[|X_{t_j\wedge \eta(\tau_N)}^{n}|^{4\gamma-2-\theta}] n^{-2} + \E[|X_{t_j\wedge \eta(\tau_N)}^{n}|^{2\gamma-\theta}] n^{-2}\\
&\qquad+  n^{-1-\frac{1}{\alpha}} \exp\big(-cn^{-(\frac{1}{\alpha}+ \rho - 1)}\big)
\big(1+\E[|X_{t_j\wedge \eta(\tau_N)}^{n}|^{\beta}]\big)
+\E[|X_{t_j\wedge \eta(\tau_N)}^{n}|^{2\gamma-1-\theta}] n^{-2} \nonumber\\
&\qquad+ \E[|X_{t_j\wedge \eta(\tau_N)}^{n}|^{4\gamma-2-\theta}] n^{-2} +\E[|X_{t_j\wedge \eta(\tau_N)}^{n}|^{2\gamma-1+\rho-\theta}] n^{-1-\frac{1}{\alpha}}\Big]. 
\end{align*}
    Note that given our choice of $\theta$, all the powers of $X^n$ in the right hand side above are positive and smaller than $\beta$. Then by using the inequality $|x|^a \leq 1+ |x|^b$ for $a \leq b$, the above can then be bounded by:
\[\frac{1}{ K } \E \big[\max_{i\in I_n} (X_{t_i\wedge \eta(\tau_N)}^{n})^{\beta} \big]+C_T + C\E\big[\int_0^T |X_{\eta(s)\wedge \eta(\tau_N)}^{n}|^{\beta}ds\,\big].\]
This shows that
\begin{equation}
\E \big[\max_{i\in I_n} |\overline{M}^{n,\tau_N}_{t_i}|^{\beta} \big]\leq
\frac{1}{ K } \E \big[\max_{i\in I_n} (X_{t_i\wedge \eta(\tau_N)}^{n})^{\beta} \big]+C_T + C\E\big[\int_0^T |X_{\eta(s)\wedge \eta(\tau_N)}^{n}|^{\beta}ds\,\big]. \label{E11}
\end{equation}

For the martingale $\widehat M^n$, we again make use of \eqref{BDG12} and proceed similarly to \eqref{GR}. That is we have
\begin{align*}
 &\big|\sum^{n-1}_{j=0} \I_{\{t_j < \tau_N\}} (\Delta \widehat{M}^n_{t_{j}})^2 \big|^\frac{\beta}{2}
=\big|\sum_{j=0}^{n-1} \I_{\{t_j < \tau_N\}}\frac{\sigma_1^4 (X_{t_j\wedge \tau_N}^{n})^{4\gamma-2}}{4\kappa_n^4} \big|(\Delta W_{t_j})^2-\Delta t \big|^2\big|^\frac{\beta}{2} \\
&\leq \frac{\sigma_1^{2\beta}}{2^\beta\kappa_0^{2\beta}} \Big|\frac{1}{ K (1-\frac{\beta}{2} )}\max_{i\in I_n} (X_{t_i\wedge \eta(\tau_N)}^{n})^{\frac{\vartheta\beta}{2-\beta}} \Big|^{(1-\frac{\beta}{2})} \\
&\qquad\times 
\Big| \big[K\big(1-\frac{\beta}{2}\big)\big]^{\frac{2}{\beta}(1-\frac{\beta}{2})}\sum^{n-1}_{j=0}   (X_{t_j\wedge \eta(\tau_N)}^{n})^{4\gamma-2-\vartheta} \big|(\Delta W_{t_j})^2-\Delta t \big|^2\Big|^\frac{\beta}{2},
\end{align*}
for some $K>0$ and $\vartheta \in ((4\gamma-2-\beta)^+, (4\gamma-2)\wedge (2-\beta))$, which is a constant that was chosen so that all powers of $X^n$ in the right hand side above are positive and smaller than $\beta$.
Let $C=\sigma_1^{2\beta}/2^\beta\kappa_0^{2\beta}$, we rewrite the right hand side above into
\begin{align*}
&C \Big|\frac{1}{C K (1-\frac{\beta}{2} )}\max_{i\in I_n} (X_{t_i\wedge \eta(\tau_N)}^{n})^{\frac{\vartheta\beta}{2-\beta}} \Big|^{(1-\frac{\beta}{2})}\\
&\qquad \qquad \times \Big| \big[CK\big(1-\frac{\beta}{2}\big)\big]^{\frac{2}{\beta}(1-\frac{\beta}{2})}\sum^{n-1}_{j=0}   (X_{t_j\wedge \eta(\tau_N)}^{n})^{4\gamma-2-\vartheta} \big|(\Delta W_{t_j})^2-\Delta t \big|^2\Big|^\frac{\beta}{2},
\end{align*}
and by applying Young's inequality with $p=2/(2-\beta)$ and $q=2/\beta$ we obtain
\begin{align*}
& \big| \sum^{n-1}_{j=0} \I_{\{t_j < \tau_N\}} (\Delta \widehat{M}^n_{t_{j}})^2 \big|^\frac{\beta}{2}\leq \frac{1}{ K }\max_{i\in I_n} (X_{t_i\wedge \eta(\tau_N)}^{n})^{\frac{\theta\beta}{2-\beta}} \\
&\qquad +C\, \frac{\beta}{2} \big[CK\big(1-\frac{\beta}{2}\big)\big]^{\frac{2}{\beta}(1-\frac{\beta}{2})}\sum^{n-1}_{j=0} (X_{t_j\wedge \eta(\tau_N)}^{n})^{4\gamma-2-\vartheta} \big|(\Delta W_{t_j})^2-\Delta t \big|^2.
\end{align*}
Again given the choice of $\vartheta$, both powers of $X^n$ in the above are positive and smaller than $\beta$, and by making use of the in equality $|x|^a \leq 1+ |x|^b$ for $0< a\leq b$, we have
\[\E\big[\big| \sum^{n-1}_{j=0} \I_{\{t_j < \tau_N\}} (\Delta \widehat{M}^n_{t_{j}})^2 \big|^\frac{\beta}{2}\big]
\leq\frac{1}{ K } \E \big[\max_{i\in I_n} (X_{t_i\wedge \eta(\tau_N)}^{n})^{\beta} \big]+C_T + C\int_0^T \E[|X_{\eta(s)\wedge \eta(\tau_N)}^{n}|^{\beta}]ds.\]
Thus for $\widehat{M}^n$ we have the following estimate
\begin{equation}
\E \big[\max_{i\in I_n} |\widehat{M}^{n,\tau_N}_{t_i}|^{\beta} \big]\leq
\frac{1}{ K } \E \big[\max_{i\in I_n} (X_{t_i\wedge \eta(\tau_N)}^{n})^{\beta} \big]+C_T + C\int_0^T \E [|X_{\eta(s)\wedge \eta(\tau_N)}^{n}|^{\beta}]ds. \label{E2}
\end{equation}

To estimate the non-square-integrable martingale $\widetilde M^n$, we apply Doob's maximal inequality and Jensen's inequality to obtain that 
\begin{align*}
\E\big[\,\max_{i\in I_n} |  \widetilde M_{t_i\wedge \tau_N}^{n}|^\beta \big]
& \leq
(1-\beta^{-1})^{-\beta} \mathbb{E}\big[\big|  \sum_{j=0}^{n-1} \I_{\{t_j<\tau_N\}}\Delta  \widetilde{M}^n_{t_j}\big|^{\beta}\big]\\
& \leq
(1-\beta^{-1})^{-\beta}n^{\beta-1}\sum_{j=0}^{n-1} \E [| \Delta \widetilde M^n_{t_j}|^{\beta}\I_{\{t_j<\tau_N\}}].
\end{align*}
We recall from \eqref{tildeM} that $\Delta \wt M^n$ is of the form $\Delta \wt{M}^n_{t_i} := \frac{\Delta M^D_{t_i}}{2(1+k\Delta t)^2}$. By using Lemma \ref{2.3} and independence, the summand can be estimated as follows, 
\begin{align}
\E[|\Delta \widetilde M^n_{t_j}|^{\beta}\I_{\{t_j<\tau_N\}}]
&\leq
C_T \E[(X_{t_j \wedge \eta(\tau_N)}^{n})^{\beta}]\Delta t ^\frac{\beta}{\alpha}
\exp (- c (\Delta t)^{-(\frac{1}{\alpha} + \rho - 1 )} )
\label{MTilde2}\\
& \leq C_T \E[(X_{t_j \wedge \eta(\tau_N)}^{n})^{\beta}]\Delta t^2  \nonumber
\end{align}
 This gives the estimate
\begin{gather}
\E\big[\max_{i\in I_n} |\widetilde M_{t_i\wedge \tau_N}^{n}|^\beta \big]
	\leq C_{T} \big(\,\sum_{j=0}^{n-1} \E[(X_{t_j \wedge \eta(\tau_N)}^{n})^{\beta}] \Delta t^2 \, \big) 
	\leq C_{T} \big(\,\sum_{j=0}^{n-1} \E[\,\max_{i \in I_j}(X_{t_i \wedge \eta(\tau_N)}^{n})^{\beta}] \Delta t^2 \, \big). \label{E3}
\end{gather}

\vskip5pt
\noindent
{\bf Step 3:} We compute the stochastic integrals against $L^2$-martingales in \eqref{ee2}. By using the set inclusion $\{s \leq \tau_N\} \subset \{\eta(s) \leq \eta(\tau_N)\}$, we can ease the computation by consider the continuous extension of the integrals and applying the Burkholder-Davis-Gundy inequality to obtain
	the following
	\begin{align*}
	\E\big[ \max_{i\in I_n} |\int_0^{t_i \wedge \tau_N}\sigma_1(X_{\eta(s)\wedge \eta(\tau_N)}^{n})^\gamma dW_s |^{\beta} \big]
	\leq
	C \E \big[ \big|\int_0^{T} |X_{\eta(s)\wedge \eta(\tau_N)}^{n}|^2 ds\big|^{\frac{\beta}{2}} \big] + CT^{\frac{\beta}{2}},
	\end{align*}
	where we have used the inequality $|x|^a \leq 1+|x|^b$ for $0<a\leq b$.	Similarly, we have 
	\begin{align*}
&\quad\E \big[ \max_{i\in I_n} |\int_0^{t_i \wedge \tau_N} \int^1_0  \sigma_2(X_{\eta(s)\wedge \eta(\tau_N)}^{n})^{\rho} z \widetilde{N}(dz,ds)|^\beta \big]\\
&\leq C\E \big[ \big| \int_0^{\tau_N} \int_0^1 ( X_{\eta(s)\wedge \eta(\tau_N)}^{n})^{2\rho} z^2 \nu(dz)ds \big|^{\frac{\beta}{2}} \Big]\leq C_{\alpha, \beta}\big(\E  \big|\int_0^{T} |X_{\eta(s)\wedge \eta(\tau_N)}^{n}|^2 ds\big|^{\frac{\beta}{2}}  + T^{\frac{\beta}{2}}\big).
\end{align*}
To this end, we can apply similar techniques to those in \eqref{GR} and \eqref{YY} from Step 2. That is we first write
\begin{align*}
& 	\big|\int^{T}_{0}   |X^{n}_{\eta(s)\wedge \eta(\tau_N)}|^2 ds \big|^\frac{\beta}{2} 
= \big|\int^{T}_{0}
\big(|X^{n}_{\eta(s)\wedge \eta(\tau_N)}|^{\beta}\big)^{\frac{2}{\beta}(1-\frac{\beta}{2})}|X^{n}_{\eta(s)\wedge \eta(\tau_N)}|^\beta ds \big|^\frac{\beta}{2}\\
& \leq   \big|\frac{1}{K (1-\frac{\beta}{2} )}\max_{i\in I_n} |X^{n}_{t_i\wedge \eta(\tau_N)}|^{\beta}\big|^{(1-\frac{\beta}{2})}\times 
\big|\big[K\big(1-\frac{\beta}{2}\big)\big]^{\frac{2}{\beta}(1-\frac{\beta}{2})}\int^{T}_{0}  |X^{n}_{\eta(s)\wedge \eta(\tau_N)}|^\beta ds \big|^\frac{\beta}{2} ,
\end{align*}
and then apply Young's inequality with $p = \frac{2}{2-\beta}$ and $q = \frac{2}{\beta}$ to obtain
\begin{align*}
\Big|\int^{T}_{0}   |X^{n}_{\eta(s)\wedge \eta(\tau_N)}|^2 ds \Big|^\frac{\beta}{2} \leq \frac{1}{K}\max_{i\in I_n} |X^{n}_{t_i\wedge \eta(\tau_N)}|^{\beta}
+ \frac{\beta}{2}\big[K\big(1-\frac{\beta}{2}\big)\big]^{\frac{2}{\beta}(1-\frac{\beta}{2})}\int^{T}_{0}  |X^{n}_{\eta(s)\wedge \eta(\tau_N)}|^\beta ds,
\end{align*}
and finally, we have
\begin{equation}\label{E4}
\E \big[ \big|\int_0^T |X_{\eta(s)\wedge \eta(\tau_N)}^{n}|^2 ds\big|^{\frac{\beta}{2}} \big]
\leq
\frac{1}{K} \E \big[ \max_{i\in I_n} |X^{n}_{t_i\wedge \eta(\tau_N)}|^{\beta} \big]
+ C_{\beta,K} \int^{T}_{0}  \E[|X^{n}_{\eta(s)\wedge \eta(\tau_N)}|^\beta] ds,
\end{equation}
where again the constant $K>0$ can be freely chosen.

\vskip3pt
\noindent {\bf Step 4:} The last term to be computed is the Poisson integral in \eqref{ee3} given by
\begin{align*}
V_{t_i} = \sigma_2\int_0^{t_i}\int_1^\infty ( X_{\eta(s)}^{n})^{\rho} z {N}(dz,ds).
\end{align*}
We first note that the process $V$ is positive. Then by applying the It\^o formula, Jensen's inequality and the fact that $\{s\leq \tau_N\}\subset \{\eta(s) \leq \eta(\tau_N)\}$, we have
\begin{align*}
\max_{i\in I_n} |V_{t_i}^{\tau_N}|^\beta \leq	\sup_{t\leq T} |V_t^{\tau_N}|^\beta
& = \int_{0}^{\tau_N} \int^\infty_1
((V^{\tau_N}_{s-} + \sigma_2 (X_{\eta(s)\wedge\eta(\tau_N)}^{n})^{\rho}z)^\beta 
- ( V_{s-}^{\tau_N})^\beta)
N(dz,ds)\\
& \leq C\int^T_0 \int^\infty_1 \,
(|V_{s-}^{\tau_N}|^\beta + |X_{\eta(s)\wedge \eta(\tau_N)}^{n}|^{\beta\rho} z^\beta)\,  N(dz,ds).
\end{align*}
By taking the expected value and noting that $\beta \in [1,\alpha)$, we have
\begin{align*}
\mathbb{E} \big[\,  \sup_{t\leq T} |V_t^{\tau_N}|^\beta\big] & \leq
C \int^T_0  (\mathbb{E} \big[|V_{s}^{\tau_N}|^\beta\big]
+ \mathbb{E} \big[|X_{\eta(s)\wedge\eta(\tau_N)}^n|^{\beta\rho}\big]) \, ds \\
& \leq 
C \int^T_0  \mathbb{E} \big[\,  \sup_{u\leq s} |V_u^{\tau_N}|^\beta\big]\, ds
+C\int^T_0 \mathbb{E} \big[|X_{\eta(s)\wedge\eta(\tau_N)}^n|^{\beta\rho}\big] \, ds  .
\end{align*}
Using the fact that $\rho < 1$, by using the inequality $|x|^a \leq 1+ |x|^b$ for $a \leq b$, Jensen's inequality and Gr\"onwall's inequality, we obtain
\begin{align}
\mathbb{E} \big[\,\sup_{t\leq T} |V_t^{\tau_N}|^\beta\big] 
& \leq C_T\Big[1+  \int^T_0 \mathbb{E} \big[|X_{\eta(s)\wedge\eta(\tau_N)}^n|^{\beta}\big] \, ds  \Big] \leq C_T\Big[  1 +  \sum_{j=0}^{n-1} \mathbb{E} \big[ \max_{i\in I_j} |X_{t_i\wedge\eta(\tau_N)}^n|^{\beta} \big] \Delta t \Big]
 \label{E5}
\end{align}
	
\vskip3pt
\noindent {\bf Step 5:}
By combining the estimates \eqref{E0}, \eqref{E1}, \eqref{E11}, \eqref{E2}, \eqref{E3}, \eqref{E4} and \eqref{E5} to obtain
\begin{align*}
\E\big[\max_{i\in I_n} ( X^{n}_{t_i\wedge \eta(\tau_N)} )^\beta\big]
&\leq C_T \Big(1 + n^{-\beta}+ \sum_{j=0}^{n-1} \E \big[\max_{i \in I_j} (X^{n}_{t_i\wedge \eta(\tau_N)})^\beta \big]\Delta t \Big)
+ \frac{4}{ K } \E \big[\max_{i\in I_n} (X_{t_i\wedge \eta(\tau_N)}^{n})^{\beta} \big],
\end{align*}
and then by choosing, for example $K=8$, we obtain
\begin{equation*}
\E\big[\max_{i\in I_n}\, ( X^{n}_{t_i\wedge \eta(\tau_N)} )^\beta\big] \leq C_T \big(1 + \sum_{j=0}^{n-1} \E\big[\max_{i \in I_j}\, (X^{n}_{t_i\wedge \eta(\tau_N)})^\beta\big] \Delta t \big).
\end{equation*}
From the above one can conclude by applying discrete Gr\"onwall's inequality and letting $N\uparrow \infty$.
\end{proof}

\begin{proof}[{\bf Proof of Lemma \ref{martingale:components}}]
Let us first record that 
\begin{gather*}
 \Delta \overline{M}^n_{t_i}   := \Delta M_{t_i}^n-\sigma_1 (X^{n}_{t_i})^\gamma  \Delta W_{t_i} \quad \mathrm{and}\quad  \Delta \widehat{M}^n_{t_i}    :=  \frac{\sigma_1^2 }{2}\frac{(X^{n}_{t_i})^{2\gamma-1}}{(1+k\Delta t)^2}((\Delta W_{t_i})^2-\Delta t ).
\end{gather*}
For the martingale $\overline{M}^n$, we apply H{\"o}lder's inequality, inequality \eqref{Mbarsq} and Lemma \ref{scheme:beta:moments} to obtain for $\beta \in [1,2]$
\begin{align*}
& \E[|\Delta \overline{M}^n_{t_{i}}|^\beta]  \leq |\E[|\Delta \overline{M}^n_{t_{i}}|^2]|^{\frac{\beta}{2}}\\
&\leq
		C \Big| n^{-2} \E[|X_{t_i}^{n}]|^{4\gamma -2} 
		+ n^{-2} \E[|X_{t_i}^{n}|^{2\gamma}]
		+ \E[|X_{t_i}^{n}|^{2\gamma-1}] n^{-1-\frac{1}{\alpha}} \exp\big(- C_2 n^{\rho-1+\frac{1}{\alpha}}\big)\\
& \quad
		+n^{-2}\E[|X_{t_i}^{n}|^{2\gamma-1}]+n^{-2} \E[|X_{t_i}^{n}|^{4\gamma-2}]
		+ \sigma_2 n^{-1-\frac{1}{\alpha}}
		\E[|X_{t_i}^{n}|^{2\gamma-1+\rho}] \Big|^{\frac{\beta}{2}} \leq C n^{-\frac{\beta}{2}-\frac{\beta}{2\alpha}}.
\end{align*}
Then, similar to \eqref{BDG12}, we apply the discrete time Burkholder-Davis-Gundy inequality and Jensen's inequality to obtain
	\[\E[|\overline{M}^n_{t_{j+1}}|^{\beta}] \leq C\E\big[ \sum_{i=0}^{n-1}|\Delta \overline{M}^n_{t_{i}}|^2  \big]^{\frac{\beta}{2}}
	\leq C n^{-\frac{\beta}{2\alpha}}.\]	
In the case when $\sigma_2 = 0$, we observe that $\E[|\Delta \overline{M}^n_{t_{i}}|^\beta]   \leq n^{\beta}$ and $\E[|\overline{M}^n_{t_{i}}|^\beta]   \leq n^{-\frac{\beta}{2}}$.	For the martingale $\widehat M^n$, we have for $\beta \in [1,2]$
	\[ \E[|\Delta \widehat{M}^n_{t_i}|^\beta]
	=
	C \E[|X^{n}_{t_i}|^{\beta(2\gamma-1)}]
	\E[||\Delta W_{t_i}|^2-\Delta t |^\beta] \leq C n^{-\beta},\]
	and also $\E[|\Delta \widehat{M}^n_{t_i}|^2] \leq n^{-2}$ since $4\gamma -2 < \alpha$.
	From this we deduce that
	\[
	\E[|\widehat{M}^n_{t_{j+1}}|^{\beta}] \leq C\E\big[ \sum_{i=0}^{n-1}|\Delta \widehat{M}^n_{t_{i}}|^2  \big]^{\frac{\beta}{2}}
	\leq C n^{-\frac{\beta}{2}}.
	\]	
which concludes the proof.
\end{proof}

\begin{proof}[{\bf Proof of Lemma \ref{alphamartingale}}]
Let us first record that 
\begin{align*}
  \Delta \wt{M}^n_{t_i} := \frac{\Delta M^D_{t_i}}{4(1+k\Delta t)^2}\quad
\mathrm{and} \quad   A_{t_i}^n \Delta t  :=  
(\Delta t)^2\Big[
		 -ak_n 
		-
	\frac{\sigma_1^2}{2}\frac{k (X^{n}_{t_i})^{2\gamma-1}  }{(1+k\Delta t)^2}\Big]
		+\frac{\E [D_{t_{i}}^- | \mathcal{F}_{t_i}]}{4(1+k\Delta t)^2},
\end{align*}
where $\Delta M_{t_i}^D := D_{t_{i}}^- - \E [D_{t_{i}}^- | \mathcal{F}_{t_i}]$. From the estimate in \eqref{Aestimate} and the inequality $e^{-x} \leq m!x^{-m}$ for $x \geq 0$ and $m\in \mathbb{N}_+$ we deduce that $\E[|A_{t_i}^n\Delta{t}|^{\beta}] \leq n^{-2\beta}$. Then by an application of Jensen's inequality, we see that
	\begin{displaymath}
	\E \big[\big|  \int_{(0,t]} A_{\eta(s)}^n ds \big|^{\beta}\big]
	\leq
	n^{\beta-1} \sum_{i=0}^n \E[|A_{t_i}^n\Delta{t}|^{\beta}] \leq C_Tn^{-\beta}.
	\end{displaymath}		
For the martingale $\widetilde M^n$, from \eqref{MTilde2} we see that $\E[|\Delta \widetilde M^n_{t_i}|^\beta] \leq  C_Tn^{-2\beta}$,
	and then by Doob's maximal inequality and Jensen's inequality together give us the following
	\[\E[| \widetilde M^n_{t_{j+1}}|^\beta] \leq C_Tn^{\beta-1} \sum_{j=0}^{n-1}\E[|\Delta \widetilde M^n_{t_j}|^\beta]
	\leq C n^{-\beta}
	\]
which concludes the proof.
\end{proof}

\begin{proof}[{\bf Proof of Corollary \ref{cor1}}]
Follows immediately from Lemma \ref{alphamartingale}.
\end{proof}

\begin{proof}[{\bf Proof of Lemma \ref{LocalLocal}}]
For $\beta \in [1,\alpha)$, we see from \eqref{continuous:extension} that by Jensen's inequality $\E[|\overline X^{n}_t - X^{n}_{\eta(t)}|^{\beta}]$ can be upper estimated by
\begin{align*}
&\E[| \overline{X}_t^{n}  - X^n_{\eta(t)}|^\beta] 
\leq \big(1+\frac{|k|^{\beta}}{\kappa_0^{\beta}}\E[|X^{n}_{\eta(t)}|^{\beta}] \big)n^{-\beta} + \sigma_1^{\beta} \E[|X^{n}_{\eta(t)}|^{\gamma\beta}] n^{-\frac{\beta}{2}} + \sigma_2^{\beta}\E[|X^{n}_{\eta(t)}|^{\beta\rho }] n^{-\frac{\beta}{\alpha}}\\
&\qquad  + \E[|\overline M^n_t - \overline M^n_{\eta(t)}|^\beta] + \E[|\widehat M^n_t -\widehat M^n_{\eta(t)}|^\beta] + \E[|\widetilde M^n_{\eta(t)}|^\beta] + \E[|\int^{\eta(t)}_0 A_{\eta(s)}^n ds|^\beta].
\end{align*}
The above can be further estimated by using H\"older's inequality, Lemma \ref{scheme:beta:moments} and Lemma \ref{martingale:components}, giving us $\E[|\overline X^{n}_t - X^{n}_{\eta(t)}|^{\beta}]\leq  C_Tn^{-\frac{\beta}{2}}$.
\end{proof}

\begin{proof}[{\bf Proof of Lemma \ref{LTE1.3}}]
	Let $x \in \mathbb{R}$, $y \in \mathbb{R} \setminus\{0\}$ with $xy \geq 0$ and $z >0$.
	By the second order Taylor's expansion for $\phi_{\delta,\varepsilon}$, it follows from \eqref{YW4} that
	\begin{align*}
	\phi_{\delta,\varepsilon}(y+xz)-\phi_{\delta,\varepsilon}(y)-xz\phi_{\delta,\varepsilon}'(y)
	&=|xz|^2 \int_{0}^{1} \theta \phi_{\delta,\varepsilon}''(y+ \theta xz)d \theta
	\leq \frac{2|xz|^2}{\log \delta} \int_{0}^{1} \frac{\theta \I_{[\varepsilon/\delta,\varepsilon]}(|y+ \theta xz|) }{|y+ \theta xz|} d \theta.
	\end{align*}
	Since $xy \geq 0$, we have $|y| \leq |y+\theta xz|$ and $\I_{[\varepsilon/\delta,\varepsilon]}(|y+ \theta xz|) \leq \I_{(0,\varepsilon]}(|y|)$.
	Hence we obtain
	\begin{align}\label{key_lem_2}
	\phi_{\delta,\varepsilon}(y+xz)-\phi_{\delta,\varepsilon}(y)-xz\phi_{\delta,\varepsilon}'(y)
	& \leq \frac{2 |xz|^2 \I_{(0,\varepsilon]}(|y|)}{ \log \delta} \left(\frac{1}{|y|} \wedge \frac{\delta}{\varepsilon}\right).
	\end{align}
	Moreover, since $xy \geq 0$, by \eqref{YW1} we have $x \phi_{\delta,\varepsilon}'(y) \geq 0$. This together with the fact that the right hand side of \eqref{key_lem_2} has $\I_{(0,\varepsilon]}(|y|)$, we obtain
	\begin{align}\label{key_lem_3}
	\phi_{\delta,\varepsilon}(y+xz)-\phi_{\delta,\varepsilon}(y)-xz\phi_{\delta,\varepsilon}'(y)
	&\leq \I_{(0,\varepsilon]}(|y|) \{\phi_{\delta,\varepsilon}(y+xz)-\phi_{\delta,\varepsilon}(y) \} \notag\\
	&=\I_{(0,\varepsilon]}(|y|) xz \int_{0}^{1}\phi_{\delta,\varepsilon}'(y+\theta xz) d \theta
	\leq \I_{(0,\varepsilon]}(|y|) |xz|.
	\end{align}
	The result then follows from \eqref{key_lem_2} and \eqref{key_lem_3}.
\end{proof}

\begin{proof}[{\bf Proof of Lemma \ref{xPrime}}]
Before proceeding, similar to Li and Mytnik \cite{LiMytnik}, we introduce the quantity
\begin{align*}
\alpha_{\nu}:=\inf\{\beta>1: \lim_{x \to 0+} x^{\beta-1} \int_{x}^{\infty}z \, \nu(dz)=0\}
\end{align*}

\noindent which represents the order of singularity of the L\'evy measure at zero. For instance, if $\nu$ is the L\'evy measure of a spectrally positive $\alpha$-stable like process, that is, if $\nu(dz) = \textbf{1}_{(0,\infty)}(z) g(z)/ z^{1+\alpha} \, dz$, with $\alpha \in [1,2]$, $g$ being a non-negative bounded and continuous function on $\R_+$, one has $\alpha_\nu = \alpha$ and the infimum is actually achieved. Also we recall from Lemma 2.1 of \cite{LiMytnik} that $\alpha_\nu \in [1,2]$ and, moreover, for any $\alpha_0>\alpha_{\nu}$, 
\begin{equation}
\lim_{x \to 0+} x^{\alpha_0-2} \int_{0}^{x} z^2 \nu(dz)=0 \quad \mathrm{and} \quad \lim_{x \to 0+} x^{\alpha_0-1} \int_{x}^{\infty}z \nu(dz)=0. \label{left:right:tail:asymptotics}
\end{equation}

To this end, note that for $x = x'$, the claimed inequality is trivially true. Therefore we suppose that $x\neq x'$. To obtain the required estimate, we consider
	\begin{align*}
	&\left| \phi_{\delta,\varepsilon}(y+xz)-\phi_{\delta,\varepsilon}(y+x'z)-(x-x')z\phi_{\delta,\varepsilon}'(y) \right|\\
	&\leq \left| \phi_{\delta,\varepsilon}(y+xz)-\phi_{\delta,\varepsilon}(y+x'z)-(x-x')z\phi_{\delta,\varepsilon}'(y+x'z) \right|
	+|x-x'||z|\left| \phi_{\delta,\varepsilon}'(y)-\phi_{\delta,\varepsilon}'(y+x'z) \right| =: A_z + B_z.
	\end{align*}	
Let $u\in (0,\infty)$. To estimate $A_z$ for $z \in (0,u)$, we apply a second order Taylor's expansion for $\phi_{\delta,\varepsilon}$ and use \eqref{YW4}. This gives
	\begin{align*}
	A_z \leq |x-x'|^2 |z|^2 \int_{0}^{1} \theta \phi_{\delta,\varepsilon}''(y+ \theta xz+(1-\theta)x'z)d \theta \leq |x-x'|^2 |z|^2 \frac{2\delta}{\varepsilon \log \delta}
\end{align*}	
while for $z \in (u,\infty)$, by the mean value theorem and \eqref{YW2}, we get 
	\begin{align*}
	A_z & \leq |x-x'| |z| \int_{0}^{1} \left|\phi'_{\delta,\varepsilon}(y+\theta xz+(1-\theta)x'z)-\phi'_{\delta,\varepsilon}(y) \right| d\theta
	\leq 2 |x-x'| |z|.
	\end{align*}
To this end, for some positive $k$ (which is chosen later) one considers the two cases $|x-x'| \leq k$ and $|x-x'|\geq k$. In the first case, take $u=1$ so that
$$
\int_0^\infty A_z \nu(dz) \leq C \left\{ |x-x'|^2 \frac{2\delta}{\varepsilon \log(\delta)} + |x-x'|\right\} \leq C_k \left\{ |x-x'|^{\alpha}\frac{2\delta}{\varepsilon \log(\delta)} + |x-x'|\right\} 
$$
\noindent for any $\alpha_0 \in [0, 2]$.
In the second case $|x-x'|\geq k$, we select $u = |x-x'|^{-1} \in (0,k^{-1})$ and remark that
$$
\int_0^{|x-x'|^{-1}} z^2 \nu(dz) = |x-x'|^{\alpha_0-2} \frac{1}{|x-x'|^{\alpha_0-2}} \int_0^{|x-x'|^{-1}} z^2 \nu(dz) \leq |x-x'|^{\alpha_0-2} \sup_{\varepsilon \in [0,k^{-1}]}I^{1}_\varepsilon
$$
\noindent with $I^{1}_\varepsilon := \varepsilon^{\alpha_0-2} \int_0^{\varepsilon} z^2 \nu(dz)$. Note that $\lim_{\varepsilon \downarrow 0}I^{1}_\varepsilon = 0$ for any $\alpha_0 > \alpha_\nu$ by Lemma 2.1 of Li and Mytnik \cite{LiMytnik}. Similarly, we have
$$
\int_{|x-x'|^{-1}}^{\infty} z \nu(dz) = |x-x'|^{\alpha_0-1} \frac{1}{|x-x'|^{\alpha_0-1}} \int_{|x-x'|^{-1}}^{\infty} z \nu(dz) \leq |x-x'|^{\alpha_0-1} \sup_{\varepsilon \in [0,k^{-1}]}I^{2}_\varepsilon
$$
\noindent with $I^{2}_\varepsilon := \varepsilon^{\alpha_0-1} \int_{\varepsilon}^{\infty} z \nu(dz)$. Again we note that $\lim_{\varepsilon \downarrow 0}I^{2}_\varepsilon = 0$ for any $\alpha_0 > \alpha_\nu$ by definition of $\alpha_\nu$. We thus conclude that for any positive constant $C$, one can pick $k$ sufficiently large, such that 
\begin{gather}
\forall \alpha_0 \in (\alpha_\nu, 2], \quad \int^\infty_0 A_z\, \nu(dz) \leq C \left\{|x-x'|^{\alpha_0}  \frac{2\delta}{\epsilon \log \delta} + |x-x'| \right\}. \label{eqA}
\end{gather}
We point out that, as previously mentioned, in the case of an spectrally positive $\alpha$-stable process we can take in the above $\alpha_\nu = \alpha$ and $\alpha_0 \in [\alpha,2]$. This is because, to obtain \eqref{eqA}, you only need the quantities $I^1_\epsilon$ and $I^2_\epsilon$ to be bounded and, in the case of the spectrally positive $\alpha$-stable process, one can check this through direct computations.

 We now deal with the term $B_z$. Let us first assume that $y x' \geq 0$. We perform a second order Taylor's expansion and employ \eqref{YW4} to obtain
\begin{align*}
	B_z & \leq |x'| |x-x'| |z|^2 \int_{0}^{1} \phi_{\delta,\varepsilon}''(y+ \theta x'z) \theta d\theta	\\
		& \leq  2 \frac{|x'||x-x'|z^2}{\log(\delta)} \int_0^1   \frac{\I_{[\varepsilon/\delta,\varepsilon]}(|y+\theta x' z|)}{|y+\theta x'z|} \, \theta d\theta
		 \leq \frac{|x'||x-x'|z^2 \I_{[0,\varepsilon)}(|y|)}{\log \delta}  \bigg( \frac{1}{|y|}\wedge \frac{\delta}{\varepsilon} \bigg)
\end{align*}

\noindent where for the last inequality we used the fact that $|y| \leq |y+\theta x'z|$ since $yx'\geq0$ and $z\geq0$.  Also, it is readily seen that $B_z \leq 2|x-x'||z|$. We thus conclude that if $y x'\geq0$, for any $u \in (0,\infty)$
$$
\int_{0}^{\infty} B_z \nu(dz) \leq |x'||x-x'| \frac{\I_{[0,\varepsilon)}(|y|)}{\log \delta}  \bigg( \frac{1}{|y|}\wedge \frac{\delta}{\varepsilon} \bigg) \int_{0}^{u} z^2 \nu(dz) + 2 |x-x'| \int_{u}^{\infty} z \nu(dz).
$$
We now treat the case $y x' < 0$. We split the $\nu(dz)$-integral into the two disjoint sets $\frac{|y|}{2|x'|}\wedge 1 < z$ and $\frac{|y|}{2|x'|}\wedge 1 \geq z$. In the case of small jumps, i.e. on the set $\frac{|y|}{2|x'|} \wedge 1 \geq z$, from the mean-value theorem and \eqref{YW4}, we obtain
\begin{align*}
B_z & = |x'||x-x'|z^2 \int_{0}^{1} \phi_{\delta,\varepsilon}''(y+ \theta x'z)d \theta	\\
& \leq 2 |x'||x-x'| z^2 \int_0^1 \frac{\I_{[\varepsilon / \delta,\varepsilon)}(|y+ \theta x'z|)}{|y+ \theta x'z|\log(\delta)} \, d\theta \leq 2 |x'||x-x'| z^2 \frac{\I_{[0,\varepsilon/2)}(|y|)}{|y|\log(\delta)}
\end{align*}
where, for the last inequality, we used the fact that $yx' < 0$ and $\frac{|y|}{2|x'|} \wedge 1 \geq z$ imply
\begin{gather*}
|y+ \theta z x'| = |y(1 - \theta z|x'||y|^{-1}|)|  \geq \frac{|y|}{2}.
\end{gather*}
Now, observe that since $y x' < 0$, one has $0\leq -\mathrm{sign}(y) x' = |x'| \leq \kappa |y|$, which combined with the previous computations yield
$$
\int_0^{\frac{|y|}{2|x'|} \wedge 1} B_z \nu(dz) \leq 2 \kappa \frac{|x-x'|}{\log(\delta)} \int_{0}^{1} z^2 \nu(dz). 
$$
For large jumps, i.e. one the set $\frac{|y|}{2|x'|} \wedge 1 \leq z$, from \eqref{YW2}, we simply note that $B_z \leq 2|x-x'||z|$ so that
$$
\int_{\frac{|y|}{2|x'|} \wedge 1}^{\infty} B_z \nu(dz) \leq 2 |x-x'| \int_{\frac{|y|}{2|x'|} \wedge 1}^{\infty} z \nu(dz) \leq C |x-x'|
$$
\noindent where we used the facts that $|x'| \leq \kappa |y|$ and $\int_{1}^{\infty} z \nu(dz)$ for the last inequality. The proof is now complete.
\end{proof}

\end{document}